\documentclass[11pt]{amsart}

\usepackage{amssymb,amsfonts,mathtools}
\usepackage[all,arc]{xy}
\usepackage{enumerate}
\usepackage{mathrsfs}
\usepackage{amsmath}
\usepackage{amsthm}
\usepackage{tikz}
\usepackage{tikz-cd}
\usepackage[utf8]{inputenc}
\usepackage{lipsum}
\usepackage{extpfeil}
\usepackage{colonequals}
\usepackage{graphicx}
\usepackage{subfig}
\usepackage{enumitem}
\usepackage{array}
\usepackage{xcolor}
\colorlet{RED}{red}
\usepackage{bm, bbm}

\newtheorem{thm}{Theorem}[section]

\newtheorem*{conj*}{Conjecture}
\newtheorem{cor}[thm]{Corollary}
\newtheorem{prop}[thm]{Proposition}
\newtheorem{lem}[thm]{Lemma}

\theoremstyle{definition}

\theoremstyle{remark}
\newtheorem{rmk}[thm]{Remark}
\newtheorem*{rmk*}{Remark}

\DeclareMathOperator{\RR}{\mathbb{R}}
\DeclareMathOperator{\ZZ}{\mathbb{Z}}
\DeclareMathOperator{\QQ}{\mathbb{Q}}
\DeclareMathOperator{\NN}{\mathbb{N}}
\DeclareMathOperator{\CC}{\mathbb{C}}
\DeclareMathOperator{\FF}{\mathbb{F}}
\DeclareMathOperator{\AAA}{\mathbb{A}}

\DeclareMathOperator{\Gal}{\operatorname{Gal}}

\DeclareMathOperator{\tr}{\operatorname{tr}}

\DeclareMathOperator{\SL}{\mathrm{SL}}
\DeclareMathOperator{\GL}{\mathrm{GL}}

\DeclareMathOperator{\ord}{\mathrm{ord}}

\DeclareMathOperator{\Real}{\operatorname{Re}}
\DeclareMathOperator{\Ima}{\operatorname{Im}}
\DeclareMathOperator{\Cl}{Cl}

\usepackage{tikz-cd}
\usepackage{float}
\usepackage[margin=1in]{geometry}
\usepackage{array}
\usepackage{comment}
\subjclass[2020]{11F66, 11M26, 11M41}

\keywords{Zero density estimates, automorphic $L$-functions, Dirichlet polynomials}

\title[Zero density estimates for $\GL_2$ automorphic $L$-functions]{Zero density estimates for $\GL_2$ automorphic $L$-functions}

\author[C. Cossaboom]{Catherine H. Cossaboom}
\address{Mathematical Institute, University of Oxford}
\email{catherine.cossaboom@maths.ox.ac.uk}
\date{}

\begin{document}
\begin{abstract} We prove zero density estimates for $L$-functions of cuspidal automorphic representations $\pi$ of $\GL_2(\AAA_{\QQ})$. We show that $N_\pi(\sigma, T) \ll T^{\frac{5}{2}(1 - \sigma) + o(1)}$, where $N_\pi(\sigma, T)$ denotes the number of zeros $\rho = \beta + i\gamma$ of $L(s,\pi)$ with $\beta \geq \sigma$ and $|\gamma| \leq T$. The key input is an extension of the Guth--Maynard argument to Dirichlet polynomials whose coefficients satisfy a weaker $\ell^q$ bound.
\end{abstract}

\maketitle

\section{Introduction}\label{intro}

Zero density estimates bound how often the zeros of an $L$-function can lie away from the critical line. These estimates produce important arithmetic consequences; they can serve as unconditional substitutes for the Riemann Hypothesis in applications in prime number theory and beyond \cite[Chapter~14]{Mont}; see also \cite[Chapter~10]{IK}.

Let $L(s, \pi)$ be the $L$-function of a unitary cuspidal automorphic representation $\pi$, and let $N_\pi(\sigma, T)$ denote the number of zeros $\rho = \beta + i\gamma$ of $L(s, \pi)$ in the rectangle $\sigma \leq \beta \leq 1$, $|\gamma| \leq T$. While the Generalized Riemann Hypothesis (GRH) would exclude all zeros with $\beta > \frac{1}{2}$, the Density Hypothesis gives the weaker quantitative prediction 
\[ N_\pi(\sigma, T) \ll_\pi T^{2(1-\sigma) + o(1)} \]
but it nevertheless suffices for many of the same applications as GRH. One might hope to obtain the conjecture by the zero detection method, which reduces the problem of bounding $N_\pi(\sigma,T)$ to the problem of estimating the frequency of large values of Dirichlet polynomials.

The large-values problem may be formulated as follows. Let 
\[ D(t) = \sum_{N < n \leq 2N} c_n n^{it}\]
be a Dirichlet polynomial with $1$-bounded coefficients, i.e. $|c_n| \leq 1$, and suppose $D(t)$ takes values of size at least $N^\sigma$ on a $1$-separated set $W \subset [0,T]$. Montgomery's large values conjecture \cite[p.~142, Conjecture~2]{MontgomeryTenLectures} predicts
\begin{equation}\label{montgomery-conjectural}
|W| \leq C(\sigma) T^{o(1)} N^{2- 2 \sigma}
\end{equation}
for some constant $C(\sigma)$ depending on $\sigma$. The central problem is to understand how closely one can approach \eqref{montgomery-conjectural} in the ranges of $N$, $T$, and $\sigma$ arising from zero detection.

Guth and Maynard \cite{GM} recently established new bounds for large values of Dirichlet polynomials with $1$-bounded coefficients. Their estimate improves the classical Montgomery--Hal\'{a}sz--Huxley bounds \cite{HuxleyLargeValues, MontgomeryLargeValues} in the range where a polynomial of length $N$ takes values of size $N^{\sigma}$ with $\frac{7}{10} \leq \sigma \leq \frac{4}{5}$. As an immediate corollary, they lower Huxley's zero density exponent for the Riemann zeta function from $12/5$ to $30/13$, the first improvement in this estimate since 1972 \cite{Huxley}.

We apply the Guth--Maynard method to $\GL_2$ automorphic $L$-functions. In the $T$-aspect, the Dirichlet polynomials in the zero detection method can have length as large as $T$, as opposed to $T^{1/2}$, which arises in the study of Dirichlet $L$-functions. We prove the following theorem.

\begin{thm}\label{mainautomorphic} Let $\pi$ be a cuspidal automorphic representation of $\GL_2(\AAA_{\QQ})$ with unitary central character, and let $L(s, \pi)$ be its standard $L$-function. Let $N_\pi(\sigma, T)$ denote the number of zeros $\rho$ of $L(s, \pi)$ with $\Real(\rho) \geq \sigma$ and $|\Ima(\rho)| \leq T$. Uniformly for $\frac{1}{2} \leq \sigma \leq 1$, we have 
\[ N_\pi(\sigma, T) \ll T^{\frac{5}{2}(1-\sigma) + o(1)}. \]
\end{thm}

For $\pi$ corresponding to a holomorphic cuspidal eigenform for $\SL_2(\ZZ)$, Ivi\'{c} \cite{Ivic} proved the uniform zero density estimate
\begin{equation}\label{ivicestimate} N_\pi(\sigma, T) \ll T^{\frac{8}{3}(1-\sigma) + o(1)}, \end{equation}
for $\frac{1}{2} \leq \sigma \leq 1$, later reproved by Yashiro \cite{Yashiro}. The same estimate holds for Hecke--Maass forms; see \cite{SankaranarayananSengupta,Tang}. The extension from full level to arbitrary fixed level is routine. In either case, the bound may also be recovered from Heath-Brown's zero density argument \cite{HB}: the pointwise divisor bounds used for the coefficients there may be replaced by $\ell^2$ bounds supplied by Rankin--Selberg theory. Thus, Theorem \ref{mainautomorphic} improves the known uniform exponent from $8/3$ to $5/2$.

Note that any cuspidal representation $\pi$ is unitarizable, and its $L$-function differs from that of the associated unitary representation only by a shift in $s$. For noncuspidal automorphic representations, the standard $L$-function factors as a product of Dirichlet $L$-functions, so the problem reduces to the degree $1$ case where stronger zero density estimates are available.

The principal technical obstacle in proving Theorem~\ref{mainautomorphic} arises for Maass forms, where the Fourier coefficients are not necessarily known to satisfy the Ramanujan conjecture, and the large-values theorem of Guth and Maynard cannot be applied directly. In contrast to classical large-values estimates, such as those of Montgomery--Hal\'{a}sz--Huxley, an $\ell^2$ bound on the coefficients does not suffice for the Guth--Maynard method. A key bottleneck in their argument is the ability to bound the additive energy of a $1$-separated set efficiently, and this step relies on an $\ell^\infty$ hypothesis. 

We overcome this by showing that the $\ell^\infty$ condition can be replaced by $\ell^q$ control with $q < \infty$. For $q \geq 8/3$, this yields the full Guth--Maynard large-values estimate. In turn, our argument clarifies how much moment control of the coefficients is required in the Guth--Maynard method.

\begin{thm}\label{momentGM} Let $q \geq 8/3$ and let $N, T \geq 2$. Let $(c_n)_{n \in \NN} \subset \CC$ and suppose that 
\[ \sum_{N < n \leq 2N} |c_n|^q \ll N^{1 + o(1)}.\]
Let $W\subset [0,T]$ be a $1$-separated set such that
\[\left|\sum_{N < n \leq 2N}c_n n^{it}\right|\ge V\]
for every $t\in W$. Then, we have 
\[ |W| \leq (NT)^{o(1)} \left( N^2 V^{-2} + N^{18/5} V^{-4} + TN^{12/5} V^{-4} \right).\]
\end{thm}

For $ 2 < q< \frac{8}{3}$, the argument also gives $q$-dependent variants of the large-values estimate; these give zero density estimates strictly between the classical exponent $8/3$ and the exponent $5/2$ obtained here once $q > \frac{5}{2}$.

We deduce Theorem \ref{mainautomorphic} from a fourth-moment bound for the \mbox{Dirichlet} coefficients $\lambda_\pi(n)$, which follows from the symmetric-square functoriality theorem of Gelbart--Jacquet \cite{GJ} and Rankin--Selberg theory \cite{JPSS}. The symmetric-cube and symmetric-fourth functoriality theorems of Kim--Shahidi \cite{KS} and Kim \cite{Kim} imply sixth- and eighth-moment bounds, respectively, but this stronger input does not improve our argument.

Theorem \ref{mainautomorphic} has a direct consequence for the distribution of Hecke eigenvalues at primes in short intervals, provided that the corresponding $L$-function has a sufficiently wide zero-free region. More precisely, suppose that, for some $c>0$, possibly depending on $\pi$, $L(s,\pi)$ has no zeros $s = \sigma + it$ in the Vinogradov--Korobov region
\begin{equation}\label{VK-hypothesis} \sigma \geq 1-\frac{c}{(\log (|t|+3))^{2/3}(\log\log (|t|+3))^{1/3}}.
\end{equation}
Then, Theorem~\ref{mainautomorphic} gives the following conditional short-interval estimate.

\begin{cor}\label{shortintervalcoefficients}
Let $\pi$ be a cuspidal automorphic representation of $\GL_2(\AAA_{\QQ})$ with unitary central character, and let $L(s, \pi)$ denote its standard $L$-function. Suppose that $L(s, \pi)$ satisfies \eqref{VK-hypothesis}. Then, for every fixed $\varepsilon>0$,
\[ \sum_{x<p\leq x+h} \lambda_\pi(p)\log p \ll_{\pi, \varepsilon} h \exp \bigl(-(\log x)^{1/4} \bigr) \]
uniformly for $x^{3/5+\varepsilon}\leq h\leq x$ as $x\to\infty$.
\end{cor}
The threshold $3/5$ reflects the new zero density exponent $5/2$; the preceding exponent $8/3$ would give the threshold $5/8$. For CM elliptic curves, the required zero-free region follows from \cite{Coleman}, and we obtain the following unconditional specialization.
\begin{cor}\label{frobeniustraces}
Let $E/\QQ$ be an elliptic curve with complex multiplication and conductor $N_E$. Write
\[ a_E(p):=p+1-\#E(\FF_p) \]
for $p \nmid N_E$. Then, for every fixed $\varepsilon>0$,
\[ \sum_{\substack{x<p\leq x+h\\ p\nmid N_E}} \frac{a_E(p)}{\sqrt p}\log p \ll_{E, \varepsilon} h \exp \bigl(-(\log x)^{1/4} \bigr) \]
uniformly for $x^{3/5+\varepsilon}\leq h\leq x$ as $x\to\infty$.
\end{cor}

Our method also produces improved zero density estimates for several other classes of $L$-functions, including certain Dedekind zeta functions, leading to improved results on the distribution of primes in Chebotarev sets over short intervals; see Section~\ref{otherdegree2}. 

The paper is structured as follows. In Section~\ref{overview}, we give an overview of the argument. In Section~\ref{automorphicprop}, we present the relevant properties of automorphic $L$-functions, including the relevant moment estimates for their Dirichlet coefficients. In Section \ref{lfunctionsandmethod}, we formulate a general class of degree $2$ \mbox{$L$-functions} with moment input, and we present the zero detection framework in Section~\ref{zeroframework}. We develop the zero detection argument in this setting in Section \ref{n0n1n2}, establishing the necessary estimates and reducing the main contribution to a large-values problem for Dirichlet polynomials. We prove Theorem \ref{momentGM} in Section \ref{mainprop} and Theorem \ref{mainautomorphic} in Section \ref{mainproof}, and Section \ref{applications} records arithmetic corollaries. In Section \ref{otherdegree2}, we derive improved zero density estimates for certain Dedekind zeta functions. Finally, in Section \ref{highdeg}, we mention the limitations of extending the present method to higher-degree $L$-functions.

\subsubsection*{Notation} We remind the reader of the following notation from \cite{GM}. We write $A \ll B$ to mean $|A| \leq CB$ for some constant $C$ and $|A| \ll_k B$ to mean $C$ may depend on a parameter $k$. We denote the condition that both $A \ll B$ and $B \ll A$ hold by $A \asymp B$, while $A \sim B$ means, more specifically, that $B < A \leq 2B$. For a large parameter $Z$, the notation $Z^{o(1)}$ denotes a subpolynomial factor; in particular, an upper bound $A\ll Z^{o(1)}B$ means that, for every $\delta > 0$, $A\ll_\delta Z^\delta B.$ A subscript in an expression such as $o_\delta(1)$ indicates that the rate of convergence may depend on the fixed parameter $\delta$. Further, $A \lessapprox B$ means that for any $\delta > 0$, $|A| \ll_\delta T^\delta B$. Similarly, $A \lessapprox_k B$ means that $A \ll_{k, \delta} T^\delta B$ for every $\delta > 0$ for a parameter $k$. We also say that a set $W \subset \RR$ is $H$-separated if $|t - t'| \geq H$ for all distinct $t, t' \in W$. Unless otherwise indicated, when an $L$-function, automorphic representation, or moment parameter is fixed, implied constants may depend on that fixed data.

\section*{Acknowledgments}

The author is grateful to James Maynard for his invaluable guidance, encouragement, and support throughout this project. She also thanks Jesse Thorner for helpful conversations. She is grateful for the support of the Marshall Aid Commemoration Commission.

\section{Overview of the Argument}\label{overview}

The first analytic difference in our setting from the case of $\zeta(s)$ is that the Dirichlet polynomials that arise from zero detection are longer. Let $s=\sigma+it$, with $\frac{1}{2}\leq \sigma\leq 1$ and $|t|\asymp T$. For a fixed $L$-function $L(s,\pi)$, the approximate functional equation gives
\begin{equation}\label{approximatefe}
L\left(s,\pi\right) = \sum_n \frac{\lambda_\pi(n)}{n^{s}} V_1\left(\frac{n}{Z}\right) + \varepsilon(\pi,t) \sum_n \frac{\overline{\lambda_\pi(n)}}{n^{1-s}} V_2\left(\frac{nZ}{\mathcal C_\pi(t)}\right),\end{equation}
where $\mathcal C_\pi(t)$ is the analytic conductor, $Z$ is a free parameter, and $V_1,V_2$ are smooth rapidly decaying weights; see \cite[Theorem~5.3]{IK}. The rapid decay of the weights shows schematically that
\[ L \left( s, \pi \right) \approx \sum_{n \ll Z} \frac{\lambda_\pi(n)}{n^{s}} + \varepsilon(\pi, t) \sum_{n \ll \mathcal C_\pi(t)/Z} \frac{\overline{\lambda_\pi(n)}}{n^{1-s}}, \qquad |t| \asymp T. \]
The approximate functional equation exhibits the reflection principle: a Dirichlet polynomial of length $Z$ with coefficients arising from $L(s,\pi)$ is associated with a dual polynomial of length $\mathcal C_\pi(t)/Z$. At a zero, the left-hand side of \eqref{approximatefe} vanishes and the two sums cancel exactly.

The polynomials produced by zero detection are mollified versions of those above. This mollifier has length $T^\eta$, where $\eta > 0$ may be taken to be arbitrarily small, so it does not affect the length scale.  Therefore, it suffices to handle Dirichlet polynomials of lengths up to $\mathcal C_\pi(t)^{1/2}$. Since $\mathcal C_\pi(t)\asymp (1+|t|)^2\asymp T^2$ for fixed $\pi$, this length is approximately $T$, as opposed to $T^{1/2}$ for $\zeta(s)$. (Note that $N$ is now bounded by a fixed power of $T$, so harmless $N^{o(1)}$ losses may be absorbed into $\lessapprox$.)

The second analytic difference is that the Ramanujan bound is not known for Maass forms in general: the coefficients are not known to satisfy $|\lambda_\pi(n)|\ll n^{o(1)}$. Before addressing this, we first consider what the argument would give if such a bound were available. Suppose that $|c_n| \leq N^{o(1)}$ for the coefficients of the Dirichlet polynomials arising from zero detection. For the purpose of this sketch, we focus on the critical case for previous estimates, which was $N \asymp T^{2/3}$ and $\sigma = 3/4$ \cite{Ivic, Tang}. We also suppress smoothing. Zero detection reduces the problem to bounding the frequency of large values of a Dirichlet polynomial
\[ D(t) = \sum_{n \sim N} c_n n^{it} \]
on a $1$-separated set $W$. At $\sigma=3/4$, Montgomery's conjecture \eqref{montgomery-conjectural} predicts
\[ |W| \lessapprox N^{1/2}, \]
and the classical large-values bound gives
\[ |W| \lessapprox N^{1/2}+T N^{-1/2}.\]
The limiting case occurs at $N\asymp T^{2/3}$, which implies $|W|\lessapprox T^{2/3}$. As a consequence of Theorem 1.1 in \cite{GM}, which requires $1$-bounded coefficients, we can replace this by
\[|W| \lessapprox N^{1/2}+N^{3/5}+T N^{-3/5},\]
and hence, at the same length $N\asymp T^{2/3}$, we have $|W|\lessapprox T^{3/5}.$ Thus, the gain in the zero density estimate is already visible at the critical case.

In the holomorphic case, Deligne's bound \cite{Deligne} gives essentially $1$-bounded pointwise control of the coefficients. The same bound is available for dihedral Maass forms through their realization as automorphic inductions of unitary Hecke characters of quadratic fields. In either case, the large-values theorem of Guth--Maynard can be applied directly using Heath-Brown's zero detection framework \cite{HB} to handle the larger Dirichlet polynomial lengths. This gives the estimate 
\[ N_\pi(\sigma,T)\ll T^{\frac{5}{2}(1-\sigma)+o(1)}. \]

For non-dihedral Maass forms, the pointwise Ramanujan bound is, in general, not known. We use instead the available moment estimates, which provide an averaged substitute: they imply that the coefficients can be large only on a sparse set. In particular, the relevant moment estimates are
\begin{equation}\label{symsquare} \sum_{n\sim N}|\lambda_\pi(n)|^2\ll N, \qquad \sum_{n\sim N}|\lambda_\pi(n)|^4\ll N^{1+o(1)},
\end{equation}
which are available due to Rankin--Selberg theory via Jacquet--Piatetski-Shapiro--Shalika \cite{JPSS} and the symmetric-square functoriality theorem of Gelbart--Jacquet \cite{GJ}. We show that the corresponding moment control can be transferred to the coefficients $c_n$ of the resulting Dirichlet polynomial in the zero detection argument.

Our strategy is to decompose the coefficients according to size by introducing dyadic index sets
\[ \mathcal{I}_B = \{ n \sim N: B < |c_n| \leq 2B \}. \]
Writing $\mathbf c_B=(c_n\mathbbm 1_{\mathcal I_B}(n))_{n\sim N}$, the fourth-moment estimate gives
\begin{equation}\label{cB} \|\mathbf c_B\|_2^2 = \sum_{n\in\mathcal I_B}|c_n|^2 \leq B^{-2}\sum_{n\sim N}|c_n|^4 \ll N^{1+o(1)}B^{-2}.
\end{equation}
Equivalently, large coefficients can occur only on a sparse set: $|\mathcal I_B|\ll N^{1+o(1)}B^{-4}.$

We ask how this sparsity interacts with the Guth--Maynard argument. After a dyadic decomposition, we may fix $B$ and a subset $W_B \subseteq W$ such that $|W| \lessapprox |W_B|$ and the Dirichlet polynomial
\[ D_B(t) := \sum_{n \in \mathcal I_B} c_n n^{it} \]
satisfies $|D_B(t)| \gtrapprox N^\sigma$ for all $t \in W_B$. Introduce the $|W_B| \times N$ matrix $M$ with entries $M_{t, n} = n^{it}$ for $t \in W_B$ and $n \sim N$. Then, we can write
\[ D_B(t) = (M \mathbf{c}_B)_t. \] Hence,
\[ N^{3/2} |W_B| = N^{2 \sigma} |W_B| \lessapprox \sum_{t \in W_B} |D_B(t)|^2 \leq \Vert M \Vert^2 \Vert \mathbf{c}_B \Vert_2^2 \]
and from \eqref{cB}, we have 
\[ N^{3/2} |W_B| \lessapprox \Vert M \Vert^2 N^{1 + o(1)} B^{-2} \Rightarrow|W_B| \lessapprox \Vert M \Vert^2 N^{-1/2} B^{-2}. \]
Thus, due to the sparsity of large coefficients, the dyadic decomposition already gives an initial saving of $B^{-2}$ at the linear-algebra stage.

On the block $\mathcal I_B$, write $\widetilde{c}_n=B^{-1} c_n$, so that $|\widetilde{c}_n|\ll 1$.  If $D_B(t)$ is large, then the normalized polynomial $\sum_{n\in\mathcal I_B} \widetilde{c}_n n^{it}$ takes values of size $N^{\sigma}/B$.  This creates a key tradeoff: the sparse support $\mathcal I_B$ gives a saving, but the large-value threshold is smaller after normalization, which introduces a competing loss.

We note that $\Vert M \Vert = s_1(M)$, where $s_1(M)$ is the largest singular value of $M$. Guth and Maynard bound $s_1(M)$ through a cubic trace estimate.  After expanding $\tr((MM^*)^3)$ and applying Poisson summation, the off-diagonal contribution is decomposed as $S_1+S_2+S_3$, where $S_j$ is the contribution from frequencies $\mathbf m=(m_1,m_2,m_3)\in\ZZ^3\setminus\{0\}$ with exactly $j$ nonzero coordinates.

These three terms behave differently with respect to the sparse support $\mathcal I_B$.  The terms $S_1$ and $S_2$ are both insensitive to the structure of the coefficients. The term $S_1$ is negligible, while $S_2$ is controlled through Heath-Brown's difference-set estimate. The term $S_3$ is the genuinely cubic part of the argument. As in \cite{GM}, we keep track of the oscillatory integrals after Poisson summation and study cancellation in
\[ R(x) :=\sum_{t\in W_B}x^{it}. \]
The resulting estimates are governed by the additive energy of $W_B$. Unlike $S_1$ and $S_2$, this part of the proof is sensitive to the lower threshold $N^\sigma/B$ for the normalized polynomial, and positive powers of $B$ enter the contribution from $S_3$. Fortunately, the loss is at most $B^{2/3}$, and our initial saving of $B^{-2}$ is sufficient to overcome this. Hence, the bound for $|W_B|$ is decreasing in $B$, and the worst case is $B \asymp 1$. We then recover the bounded-coefficient bound
\[ N_\pi(\sigma,T)\ll T^{\frac{5}{2}(1-\sigma)+o(1)}, \qquad 7/10\le \sigma\le 4/5\]
for non-dihedral Maass forms. The critical length in the argument shifts to $N = T^{5/8}$ from $T^{2/3}$.

Establishing our zero density estimate requires only limited use of the sparsity of the sets $\mathcal I_B$. Indeed, in Theorem \ref{mainautomorphic}, we have already reached the exponent available under the Ramanujan bound. It is nevertheless natural to ask whether sparsity can be exploited more within the Guth--Maynard argument to reduce the moment information required in Theorem \ref{momentGM}. We discuss this further in Remark~\ref{longrmk}.

\section{Automorphic $L$-Functions}\label{automorphicprop}

We collect here the relevant facts about automorphic $L$-functions; see, for example, \cite{Gelbart}, \cite{GodementJacquet} and \cite[Sections~5.11 and~5.12]{IK}.

Given cuspidal automorphic representations $\Pi_j$ of $\GL_{n_j}(\AAA_{\QQ})$ for $1\leq j\leq r$, we write $\Pi_1\boxplus\cdots\boxplus\Pi_r$ for their isobaric sum, an automorphic representation of $\GL_n(\AAA_{\QQ})$, where $n=n_1+\cdots+n_r$. The corresponding standard $L$-function factors as
\[ L(s,\Pi_1\boxplus\cdots\boxplus\Pi_r) = \prod_{j=1}^r L(s,\Pi_j). \]

We use the standard normalization with respect to unitary twists. For a cuspidal unitary automorphic representation $\pi$, let $\omega_\pi$ denote its central character. For $\tau\in\RR$, write $\pi(\tau):=\pi\otimes|\det|^{i\tau}$ so that
\[ L(s,\pi(\tau))=L(s+i\tau,\pi), \qquad \omega_{\pi(\tau)}=\omega_\pi|\cdot|^{2i\tau}.\]
Let $\mathcal{F}$ denote the set of cuspidal unitary automorphic representations of $\GL_2(\AAA_{\QQ})$ whose central characters are trivial on the diagonally embedded positive reals. Every cuspidal unitary automorphic representation is a unique unitary twist of an element of $\mathcal{F}$.

\subsection{Local factors and the functional equation}

Let $\pi \in \mathcal{F}$ be fixed. We use the standard normalization of the $L$-function, in which the functional equation is centered at $s=\frac{1}{2}$, and write \[ L(s, \pi) = \sum_{n=1}^\infty \lambda_\pi(n) n^{-s} \qquad (\Real s > 1), \] where $\lambda_\pi(n)$ denotes the $n$th Dirichlet coefficient of $L(s, \pi)$. As $\pi$ is cuspidal, $L(s, \pi)$ is entire. All asymptotic notation in this section will be understood with $\pi$ fixed. First, note that $\pi$ factors as a restricted tensor product 
\[ \pi = \pi_\infty \otimes \left( \bigotimes_p' \pi_p \right)\]
where $\pi_\infty$ is a smooth admissible representation of $\GL_2(\RR)$ and $\pi_p$ is a smooth admissible representation $\GL_2(\QQ_p)$ for every prime $p$. Moreover, $\pi_p$ is unramified for all but finitely many $p$. For every prime $p$, the local $L$-function at $p$ may be written as
\[ L(s, \pi_p) := (1 - \alpha_{1,\pi}(p) p^{-s})^{-1} (1 - \alpha_{2,\pi}(p) p^{-s})^{-1} \]
where we refer to the parameters $\alpha_{j, \pi}(p)$ as the Satake parameters, allowing some of them to be zero at ramified primes. For $j \in \{ 1, 2 \}$, they satisfy 
\begin{equation}\label{satakebound} |\alpha_{j, \pi}(p)| \leq p^{7/64} \end{equation}
by \cite[Appendix 2]{Kim}. In turn, $L(s, \pi)$ has an absolutely convergent Euler product 
\[ L(s, \pi) := \prod_p L(s, \pi_p) = \prod_p (1 - \alpha_{1,\pi}(p) p^{-s})^{-1} (1 - \alpha_{2,\pi}(p) p^{-s})^{-1} \qquad (\Real(s) > 1).\]
Moreover, there exist complex Langlands parameters $(\kappa_{j, \pi})_{j=1}^2$ which form the local $L$-function
\[ L(s, \pi_\infty) := \pi^{-s} \Gamma \left( \frac{s + \kappa_{1, \pi}}{2} \right) \Gamma \left( \frac{s + \kappa_{2, \pi}}{2} \right). \]

Let $q_\pi \in \NN$ denote the arithmetic conductor of $\pi$, which satisfies $\alpha_{1,\pi}(p) \alpha_{2,\pi}(p) \neq 0$ for $p \nmid q_\pi$. The completed $L$-function 
\[ \Lambda(s, \pi) = q_\pi^{s/2} L(s, \pi_\infty) L(s, \pi) \]
is entire of order $1$.

Let $\widetilde{\pi} \in \mathcal{F}$ be the dual of $\pi$, which means that $\widetilde{\pi}$ satisfies
\[ \{ \alpha_{j, \widetilde{\pi}}(p): 1 \leq j \leq 2 \} =  \{ \overline{\alpha_{j, \pi}(p)}: 1 \leq j \leq 2 \} , \quad \{ \kappa_{j, \widetilde{\pi}}: 1 \leq j \leq 2 \} = \{ \overline{\kappa_{j, \pi}} : 1 \leq j \leq 2 \}, \quad q_{\widetilde{\pi}} = q_\pi. \]
Then, there exists a root number $\varepsilon(\pi)$ of modulus $1$ such that the functional equation
\[ \Lambda(s, \pi) = \varepsilon(\pi) \Lambda(1-s, \widetilde{\pi}) \]
holds for all $s \in \CC$. The nontrivial zeros of $L(s, \pi)$ are the zeros of $\Lambda(s, \pi)$, and they lie in the critical strip $0 \leq \Real(s) \leq 1$.

\subsection{Archimedean parameters and the analytic conductor}\label{param} 
The archimedean parameters take explicit forms for $\pi \in \mathcal{F}$. For a convenient reference, see \cite{HumphriesAutomorphicNotes}. Since $\omega_\pi$ is trivial on the diagonally embedded positive reals, it is the idelic lift of a Dirichlet character. By the classical correspondence, $\pi$ either arises from a holomorphic modular form of weight $k \geq 2$ or a Hecke--Maass newform of weight $0$ or $1$.\footnote{Here, the weight-$1$ Hecke--Maass case includes forms $f(z) = y^{1/2} F(z)$ arising from holomorphic weight-$1$ newforms $F$. Such forms necessarily have odd nebentypus.} At the archimedean place, these correspond to discrete-series representations and principal-series representations, respectively.

If $\pi$ is associated to a holomorphic modular form of weight $k \geq 2$, we have
\[ \kappa_{1, \pi} = \frac{k-1}{2}, \qquad \kappa_{2,\pi} = \frac{k+1}{2}.\]
In particular, $\Real(\kappa_{j, \pi}) \geq \frac{1}{2}$ for each $j \in \{ 1, 2 \}$. Now, suppose that $\pi$ is associated to a Maass cusp form. Let $k_\pi \in \{0,1\}$ denote its weight and let $\lambda=1/4+t_\pi^2$ be its Laplacian eigenvalue. If $k_\pi = 0$, let $\vartheta\in \{0,1 \}$ denote its parity parameter, with $\vartheta=0$ in the even case and $\vartheta=1$ in the odd case; when $k_\pi =1$, set $\vartheta = 0$. Then, we have
\[ \kappa_{1,\pi}=k_\pi+ (1-k_\pi)\vartheta + it_\pi, \qquad \kappa_{2,\pi}=(1-k_\pi)\vartheta-it_\pi.\]
By the positivity of the Laplacian eigenvalue, either $t_\pi \in\RR$ or $t_\pi \in i(-1/2, 1/2)$. In either case, $\Real(\kappa_{j, \pi}) > -\frac{1}{2}$.

For $s\in\CC$, the analytic conductor at $s$ is
\begin{equation}\label{conductordef} \mathcal{C}(\pi,s) := q_\pi\prod_{j=1}^{2}\left(3+\left|s+\kappa_{j,\pi}\right|\right).
\end{equation}
We will mainly use this conductor on vertical lines. Thus, for $t\in\RR$, we write
\[ \mathcal{C}_\pi(t) := \mathcal{C}(\pi, it) = q_\pi\prod_{j=1}^{2}\left(3+\left|it+\kappa_{j,\pi}\right|\right).\]
Since $\pi$ is fixed,
\begin{equation}\label{conductoreqn} \mathcal{C}_\pi(t)\asymp (1+|t|)^2.
\end{equation}

\subsection{Coefficient moments}

We will need Rankin--Selberg theory to bound the coefficient moments. Let $\pi, \pi' \in \mathcal{F}$, and let $r_{\pi\times\pi'}=1$ if $\pi'\simeq\widetilde{\pi}$, and $0$ otherwise. We define the Rankin--Selberg convolution $L(s, \pi \times \pi')$ as follows. For every prime $p \nmid q_\pi q_{\pi'}$, we define the local $L$-function
\[ L(s, \pi_p \times \pi'_p) := \prod_{j=1}^2 \prod_{j'=1}^2 (1 - \alpha_{j, \pi}(p) \alpha_{j', \pi'}(p) p^{-s})^{-1}. \]
In \cite{JPSS}, Jacquet, Piatetski-Shapiro, and Shalika define local factors at the ramified and archimedean places and a conductor $q_{\pi \times \pi'}$ such that the $L$-function
\[ L(s, \pi \times \pi') := \prod_p L(s, \pi_p \times \pi'_p)  =: \sum_{n=1}^\infty \lambda_{\pi \times \pi'}(n) n^{-s}\]
converges absolutely for $\Real(s) > 1$ and such that the completed $L$-function
\[ \Lambda(s, \pi \times \pi') = (s(1-s))^{r_{\pi \times \pi'}} q_{\pi \times \pi'}^{s/2} L_\infty(s, \pi \times \pi') L(s, \pi \times \pi')\]
is entire of order $1$ and satisfies the functional equation
\[ \Lambda(s, \pi \times \pi') = \varepsilon(\pi \times \pi') \Lambda(1-s, \widetilde{\pi} \times \widetilde{\pi}')\]
for some root number $\varepsilon(\pi \times \pi')$ of modulus $1$. It follows that $L(s,\pi\times\pi')$ is holomorphic except possibly for a pole at $s=1$, and this pole occurs precisely when $\pi'\simeq \widetilde{\pi}$.

Specializing $\pi'=\widetilde{\pi}$ gives the classical Rankin--Selberg square series. After omitting the finitely many Euler factors at primes $p\mid q_\pi$, one has
\[ L^{(q_\pi)}(s,\pi\times\widetilde{\pi}) = \zeta^{(q_\pi)}(2s) \sum_{\substack{n\geq 1\\(n,q_\pi)=1}} |\lambda_\pi(n)|^2 n^{-s}. \]
For $(n,q_\pi)=1$, comparing coefficients gives
\begin{equation}\label{coeffineq} |\lambda_\pi(n)|^2 \leq \sum_{d^2\mid n}|\lambda_\pi(n/d^2)|^2 = \lambda_{\pi\times\widetilde{\pi}}(n).
\end{equation}
Since the coefficients $\lambda_{\pi\times\widetilde{\pi}}(n)$ are nonnegative and $L(s,\pi\times\widetilde{\pi})$ has a simple pole at $s=1$, a standard smoothed Perron argument gives
\begin{equation}\label{2ndmoment} \sum_{\substack{n\leq x}} |\lambda_\pi(n)|^2 \leq \sum_{\substack{n\leq x}} \lambda_{\pi\times\widetilde{\pi}}(n) \ll_\pi x.
\end{equation}
The contribution from the finitely many ramified primes is absorbed into the implied constant.

We now establish the required fourth-moment estimate for the coefficients $\lambda_\pi(n)$. By Gelbart--Jacquet \cite{GJ}, the adjoint lift $\operatorname{Ad}\pi\simeq \operatorname{Sym}^2\pi\otimes\omega_\pi^{-1}$ is automorphic and $L(s,\pi\times\widetilde{\pi}) = L(s,\mathbf 1\boxplus \operatorname{Ad}\pi)$. Let $\Pi=\mathbf 1\boxplus\operatorname{Ad}\pi$. By \eqref{coeffineq}, we then have, for $(n,q_\pi)=1$,
\[ \sum_{\substack{n\leq x\\(n,q_\pi)=1}}|\lambda_\pi(n)|^4 \leq \sum_{\substack{n\leq x\\(n,q_\pi)=1}}|\lambda_{\pi \times \widetilde{\pi}}(n)|^2 = \sum_{\substack{n\leq x\\(n,q_\pi)=1}}|\lambda_\Pi(n)|^2. \]
Since $\operatorname{Ad} \pi$ need not be cuspidal, write $\Pi=\Pi_1\boxplus\cdots\boxplus\Pi_r$ with each $\Pi_j$ cuspidal. Then
\[ L(s,\Pi\times\widetilde{\Pi}) = \prod_{j=1}^r\prod_{j'=1}^r
L(s,\Pi_j\times\widetilde{\Pi}_{j'}),\]
so $L(s,\Pi\times\widetilde{\Pi})$ has at most a pole of finite order at $s=1$, which in general is not simple. Moreover, away from the finitely many ramified primes, the Dirichlet coefficients of $L(s,\Pi\times\widetilde{\Pi})$ dominate $|\lambda_\Pi(n)|^2$. As above, a standard smoothed Perron argument therefore gives
\[ \sum_{\substack{n\leq x \\ (n,q_\pi) = 1}}|\lambda_\Pi(n)|^2 \ll_\pi x(\log x)^{O_\pi(1)} = x^{1+o(1)}.\]
Therefore, we get
\begin{equation}\label{4thmoment}
\sum_{n\leq x}|\lambda_\pi(n)|^4 \ll_\pi x^{1+o(1)}.
\end{equation}
As above, the finitely many ramified Euler factors are absorbed into the
implied constant.

\begin{rmk}
For completeness, suppose that $\pi = \omega_1 \boxplus \omega_2$ is a noncuspidal unitary automorphic representation of $\GL_2(\AAA_{\QQ})$. Each unitary Hecke character $\omega_j$ is of the form $\omega_j = \chi_j |\cdot|^{i \tau_j}$ with $\tau_j \in \RR$ and $\chi_j$ a primitive Dirichlet character, and hence, $L(s, \omega_j) = L(s + i \tau_j, \chi_j)$. It follows that 
\[ N_{\pi}(\sigma, T) \leq N_{\chi_1}(\sigma, T + O_\pi(1)) +  N_{\chi_2}(\sigma, T + O_\pi(1)).\] 
We thus require only a zero density estimate for a Dirichlet $L$-function. From \cite{GM}, we claim
\begin{equation}\label{degoneZDE} N_{\chi}(\sigma,T) \ll T^{30/13(1-\sigma)+o(1)}. \end{equation}
Indeed, since $\chi$ is fixed, the analytic conductor of $L(s,\chi)$ at height $T$ is of the same order as that of $\zeta(s)$. The argument of \cite{GM} does not rely on any special properties of the coefficients of $\zeta(s)$ beyond a pointwise bound, and it applies equally to any fixed Dirichlet $L$-function. Thus, we obtain the exponent $5/2$ uniformly for any irreducible generic unitary automorphic representation of $\GL_2(\AAA_{\QQ})$.\end{rmk}

\section{Admissible degree $2$ $L$-functions}\label{lfunctionsandmethod}

We set up the zero detection argument for degree $2$ $L$-functions with a prescribed coefficient-moment bound. Our $L$-functions will essentially just be automorphic $L$-functions, but it will be convenient for us to formulate the argument under the precise hypotheses it uses.\footnote{Our axioms are modeled on the Selberg class, but we replace the pointwise Ramanujan bound and approximate it with a weaker moment estimate to include Maass forms unconditionally. Conjecturally, the primitive degree $2$ members of the Selberg class are precisely the standard cuspidal $\GL_2$ automorphic $L$-functions; thus, though the present formulation formally defines a larger class, one might expect it to encompass roughly the same objects.} Throughout this section, $q \geq 2$ denotes a fixed moment parameter. We write $Q$ for the arithmetic conductor to avoid confusion with $q$.

\subsection{Admissible degree $2$ $L$-functions with $q$th-moment input}\label{properties}

Let $q \geq 2$ be a real number. We say that \[ L(s) = \sum_{n=1}^\infty a_n n^{-s} \qquad a_n \in \CC, \ a_1 = 1 \] is an admissible $L$-function of degree $2$ with $q$th-moment input if there exists an integer $Q \geq 1$, known as the arithmetic conductor, Satake parameters $\alpha_1(p), \alpha_2(p) \in \CC$ for each prime $p$, and Langlands parameters $\kappa_1, \kappa_2 \in \CC$ satisfying the following conditions.

\begin{enumerate}
    \item \textit{Euler product:} The Dirichlet series and degree $2$ Euler product converge absolutely for $\Real(s) > 1$, and 
    \[ L(s) = \sum_{n = 1}^\infty a_n n^{-s} = \prod_p (1 - \alpha_1(p) p^{-s})^{-1} (1- \alpha_2(p) p^{-s})^{-1}. \]
    To avoid poles in $\Real(s) > 1$, we further require $|\alpha_j(p)| < p$ for all $j \in \{ 1, 2 \}$ and $p$ prime. Moreover, if $p \nmid Q$, we require $\alpha_1(p) \alpha_2(p) \neq 0$.
    \item \textit{Analyticity:} There exists a nonnegative integer $m$ such that $(s-1)^m L(s)$ is an entire function of finite order.
    \item \textit{Functional equation:} $\kappa_1$ and $\kappa_2$ satisfy $\Real(\kappa_j) > -1$ for $j \in \{ 1, 2 \}$\footnote{A common simplifying hypothesis, used for example in \cite[Chapter~5]{IK}, is that the multiset of archimedean parameters is stable under complex conjugation. We do not impose this condition since it is not preserved by unitary twists. Instead, the dual $L$-function is assigned the conjugate archimedean parameters, and compatibility is encoded by the functional equation. } and are such that, after defining
    \[ \Lambda(s) = Q^{s/2} \pi^{-s} \Gamma \left( \frac{s+ \kappa_1}{2} \right) \Gamma \left( \frac{s+ \kappa_2}{2} \right) L(s), \]
    the completed $L$-function $\Lambda(s)$ is holomorphic on $\CC\setminus\{0,1\}$ and satisfies 
    \[ \Lambda(s) = \varepsilon\, \overline{\Lambda(1-\overline{s})} \]
    for some root number $\varepsilon\in\CC$ with $|\varepsilon|=1$. 
    \item \textit{Coefficient $q$th-moment bound:} The coefficients $a_n$ satisfy, as $x \rightarrow \infty$, \begin{equation}\label{qmoment} \sum_{n \leq x} |a_n|^q \ll x^{1+o(1)}.\end{equation}
\end{enumerate}

\begin{rmk} For automorphic $L$-functions, the explicit description of the archimedean parameters gives the stronger bound $\Real(\kappa_j) > -\frac{1}{2}$. For admissible $L$-functions, we follow \cite[Chapter~5]{IK} and assume only $\Real(\kappa_j) > -1$, which ensures that the gamma factors have no poles in $\Real s \geq 1$. This condition does not preclude gamma-factor poles from lying in the critical strip; these poles are necessarily cancelled by corresponding trivial zeros of $L(s)$ due to the holomorphicity of $\Lambda(s)$ on $\CC \setminus \{ 0, 1 \}$. There are only finitely many such zeros, and their contribution to the zero counts considered below is therefore $O_L(1)$.

We do not impose the stronger Selberg-class condition $\Real(\kappa_j) \geq 0$, which would remove such complications. Indeed, in the even weight-zero Maass case, an exceptional eigenvalue $\lambda = \frac{1}{4} - \nu^2$ with $0 < \nu < \frac{1}{2}$ gives archimedean parameters $-\nu$ and $\nu$. Thus, this stronger condition would rule out exceptional Maass forms and, in turn, imply Selberg's eigenvalue conjecture for these forms, which is not known.
\end{rmk}

In Section \ref{automorphicprop}, we have seen that $L(s, \pi)$ for any $\pi \in \mathcal{F}$ satisfies (1) through (4) given $q = 4$. These properties are preserved under unitary twists, so every $L$-function in Theorem \ref{mainautomorphic} is an admissible degree $2$ $L$-function with fourth-moment input. We now prove a zero density estimate for admissible degree $2$ $L$-functions, which will imply Theorem \ref{mainautomorphic}.

\begin{thm}\label{main} Let $L(s)$ be an admissible degree $2$ $L$-function with $q$th-moment input for $q \geq 8/3$. Let $N_L(\sigma, T)$ denote the number of zeros $\rho$ of $L(s)$ with $\Real(\rho) \geq \sigma$ and $|\Ima(\rho)| \leq T$. We have that 
\[ N_L(\sigma, T) \ll T^{\alpha(\sigma)(1- \sigma) + o(1)} \]
with
\begin{equation}\label{alpha} \alpha(\sigma)
= \begin{cases}
\displaystyle \frac{4}{3-2\sigma}, & \frac{1}{2}\leq\sigma\leq\frac{7}{10},\\[6pt]
\displaystyle \frac{5}{2}, & \frac{7}{10}\leq\sigma\leq\frac{4}{5},\\[6pt]
\displaystyle \frac{2}{\sigma}, & \frac{4}{5}\leq\sigma\leq1.
\end{cases} \end{equation}
\end{thm}

Compared with \eqref{ivicestimate}, the new feature of Theorem \ref{main} is the exponent $5/2$ in the middle range $\frac{7}{10} \leq \sigma \leq \frac{4}{5}$; the classical estimates are retained in the remaining ranges. Moreover, the exponent $\alpha(\sigma)$ is independent of $q$ for $q \geq 8/3$; in particular, $q = 4$ suffices to obtain Theorem \ref{mainautomorphic}.

\subsection{Standard consequences}
We record several consequences of the preceding hypotheses that will be used repeatedly throughout the zero detection argument.

\subsubsection{The analytic conductor} An identical argument to that of Section \ref{param} shows that the analytic conductor $\mathcal{C}_L(t)$, defined identically to \eqref{conductordef} for a more general $L(s)$, is $\asymp_L (1 + |t|)^2$ for all such degree $2$ $L$-functions, where the dependence on $L$ includes the fixed arithmetic conductor $Q$.

\subsubsection{Polynomial growth and zero counting} We first establish polynomial growth in fixed vertical strips and the corresponding local zero-counting estimates.

\begin{lem}\label{polygrowth}
Let $L(s)$ be an admissible degree $2$ $L$-function, and let $m\geq 0$ denote the order of its pole at $s=1$. Then, $(s-1)^mL(s)$ has polynomial growth in every fixed vertical strip. In particular, we have
\begin{equation}\label{halfbound} L\left(\frac{1}{2}+it\right) \ll (1+|t|)^{1/2}. \end{equation}
\end{lem} 

\begin{proof}
The Phragm\'en--Lindel\"of argument in \cite[Lemma~5.2]{IK} applies to \mbox{$(s-1)^mL(s)$}: absolute convergence and the functional equation together with Stirling's formula give polynomial growth on the boundary of a sufficiently wide strip, while finite order allows interpolation across it. The argument for Heath-Brown's convexity bound \cite{HBconvexity} removes the $\varepsilon$ loss on the critical line.\footnote{For standard automorphic $\GL_2$ $L$-functions, subconvexity bounds are known; see \cite{MichelVenkatesh}. No subconvexity input is needed here.}
\end{proof}

\begin{lem}\label{standardzerocount}
Let $L(s)$ be an admissible degree $2$ $L$-function. For $T\geq2$ and every interval $I\subset[-2T,2T]$, we have
\begin{equation}\label{localzerocount}
\#\{\rho:0\leq\Real(\rho)\leq1,\ \Ima(\rho)\in I\} \ll (|I|+1)\log T,
\end{equation}
where zeros are counted with multiplicity. In particular,
\begin{equation}\label{zerocount} \#\{\rho:0\leq\Real(\rho)\leq1,\ |\Ima(\rho)|\leq T\} \ll T\log T.
\end{equation}
Moreover, every zero $\rho$ in the critical strip with $|\Ima(\rho)|\leq T$ has multiplicity $O(\log T)$.
\end{lem}

\begin{proof}
\eqref{localzerocount} and, in turn, \eqref{zerocount} follow from Jensen's formula and Lemma~\ref{polygrowth}; compare to \cite[(5.27)]{IK}. The multiplicity bound follows by applying \eqref{localzerocount} near the ordinate of the zero.
\end{proof}

\subsubsection{Reciprocal coefficients}
The reciprocal coefficients will play a central role in the zero detection argument, where a truncation of $1/L(s)$ is used to construct the mollifier. For $\Real s>1$, define $b_n$ to be the coefficients of the reciprocal $L$-function \begin{equation}\label{bn} \frac{1}{L(s)} =: \sum_{n=1}^\infty b_n n^{-s} \end{equation} where $L(s)$ is an admissible degree $2$ $L$-function with $q$th moment input. The coefficients $b_n$ are multiplicative, and comparison of the local Euler factors gives
\begin{align}\label{descriptionofbn}
b_{p^k} = \begin{cases} -a_p=-(\alpha_1(p)+\alpha_2(p)) &\textrm{ for } k = 1, \\ a_p^2-a_{p^2}=\alpha_1(p)\alpha_2(p) &\textrm{ for } k = 2, \\
0 &\textrm{ for all } k\geq3. \end{cases}
\end{align}

\begin{lem}\label{bn-properties}
Let $b_n$ be defined as in \eqref{bn}. For every fixed $\delta>0$,
\begin{equation}\label{bnabsconv} \sum_{n\geq1}\frac{|b_n|}{n^{1+\delta}} \ll_{\delta}1.
\end{equation}
\end{lem}

\begin{proof}
The absolute convergence of the Euler product implies that, for every fixed $\delta>0$,
\[ \sum_p\frac{|\alpha_1(p)|+|\alpha_2(p)|}{p^{1+\delta}}<\infty \Rightarrow \sum_p\frac{|b_p|}{p^{1+\delta}}<\infty. \]
Moreover,
\[\frac{|b_{p^2}|}{p^{2+2\delta}}= \frac{|\alpha_1(p)\alpha_2(p)|}{p^{2+2\delta}}\leq\frac{1}{4}\left( \frac{|\alpha_1(p)|+|\alpha_2(p)|}{p^{1+\delta}}\right)^2 \Rightarrow\sum_p\frac{|b_{p^2}|}{p^{2+2\delta}}<\infty.\]
Since $b_{p^k}=0$ for $k\geq3$, multiplicativity now gives \eqref{bnabsconv}.
\end{proof}

The following lemma actually establishes a stronger statement from which Lemma \ref{bn-properties} could be recovered. We record Lemma \ref{bn-properties} separately because it follows directly from absolute convergence of the Euler product without invoking the coefficient moment hypothesis.

\begin{lem}\label{bn-bound} With $b_n$ defined as in \eqref{bn}, we have
\[ \sum_{n \sim N} |b_n|^q \ll N^{1 + o(1)}.\]
\end{lem}

\begin{proof}
For $\delta$ sufficiently large, absolute convergence and the multiplicativity of $b_n$ give
\begin{equation}\label{bnq} \sum_{n=1}^\infty \frac{|b_n|^q}{n^{1+\delta}} = \prod_p \left( 1+\frac{|a_p|^q}{p^{1+\delta}} +\frac{|a_p^2-a_{p^2}|^q}{p^{2+2\delta}} \right),\end{equation}
by \eqref{descriptionofbn}. We show that the Euler product converges for any $\delta>0$. From $1+x\leq e^x$, convergence of
\[\sum_p\frac{|a_p|^q}{p^{1+\delta}}+\sum_p\frac{|a_p^2-a_{p^2}|^q}{p^{2+2\delta}}\]
implies convergence of the Euler product, with
\[\prod_p\left(1+\frac{|a_p|^q}{p^{1+\delta}}+\frac{|a_p^2-a_{p^2}|^q}{p^{2+2\delta}}\right)\leq\exp\left(\sum_p\frac{|a_p|^q}{p^{1+\delta}}+\sum_p\frac{|a_p^2-a_{p^2}|^q}{p^{2+2\delta}}\right).\]
It therefore suffices to show that the two prime sums converge. For the first sum, summing over dyadic $P$ and using \eqref{qmoment}, we have
\[\sum_p\frac{|a_p|^q}{p^{1+\delta}} \ll \sum_{P\ {\rm dyadic}} P^{-1-\delta}\sum_{p \sim P}|a_p|^q \ll \sum_{P\ {\rm dyadic}}P^{-\delta+o(1)} \ll_\delta 1.\]
For the second sum, using $|a_p^2-a_{p^2}|^q\ll |a_p|^{2q}+|a_{p^2}|^q$,
we similarly obtain
\[\sum_p\frac{|a_p|^{2q}}{p^{2+2\delta}}\ll\sum_{P\ {\rm dyadic}} P^{-2-2\delta} \left(\sum_{p \sim P}|a_p|^q\right)^2 \ll \sum_{P\ {\rm dyadic}}P^{-2\delta+o(1)} \ll_\delta1, \]
while
\[\sum_p\frac{|a_{p^2}|^q}{p^{2+2\delta}} \ll \sum_{P\ {\rm dyadic}} P^{-2-2\delta} \sum_{P^2<n\le4P^2}|a_n|^q \ll \sum_{P\ {\rm dyadic}}P^{-2\delta+o(1)} \ll_\delta1. \]
Thus, the Euler product converges for every $\delta > 0$, and the left-hand side of \eqref{bnq} is $O_\delta(1)$. Since $\delta > 0$ is arbitrary, the result for $b_n$ follows.
\end{proof}

\section{The zero detection framework}\label{zeroframework}

Fix $0<\eta<1/2$ and let $X=T^\eta$. Recall that $b_n$ are the coefficients of the reciprocal $L$-function $\frac{1}{L(s)}$, and let $M_X(s) = \sum_{n \leq X} b_n n^{-s}$. Further, write
\[ L(s) M_X(s) = \sum_{n=1}^\infty C_n n^{-s}. \]
Then, $C_1 = 1$ and $C_n = 0$ for $2 \leq n \leq X$. To motivate the zero detection argument, first consider a maximally unbalanced approximate functional equation. Taking $Z = \mathcal C_L(t) T^\delta$ in \eqref{approximatefe}, where $\delta > 0$  is arbitrarily small, makes the dual sum negligible and allows us to write
\[ L(s) \approx \sum_{n \lessapprox \mathcal C_L(t)} \frac{a_n}{n^s}. \]
After multiplying by the mollifier $M_X(s)$, this becomes
\[ L(s)M_X(s) \approx 1+\sum_{X<n\lessapprox X\mathcal C_L(t)} \frac{C_n}{n^s}. \]
At a zero $\rho = \beta + i \gamma$, the tail must cancel the constant term, so dyadic decomposition produces a Dirichlet polynomial that is large at $t = - \gamma$. This approach allows us to detect all zeros through detecting large values of Dirichlet polynomials, but it produces polynomials of length as large as $T^\eta \mathcal C_L(t)$, where we will eventually take $\eta$ to be arbitrarily small.

To shorten the polynomials, we instead keep both sums in \eqref{approximatefe} and reduce to the length $Z = \mathcal C_L(t)^{1/2}$. The tradeoff is that the two sums do not behave symmetrically under mollification. The mollifier cancels the initial coefficients for the left-hand sum, and we again reduce to a large-values problem; this forms the Type I contribution. No analogous cancellation occurs for the right-hand sum. The corresponding Type II contribution takes the form of a shifted contour integral and is treated by a different analytic input: a mean-value estimate for $L(s)$ on the critical line.

We now make this decomposition precise, following here the methodology of \cite{HB}. Let $Y\geq 4$. For $s = \sigma + it$ with $\sigma > -1$, we have the Mellin transform
\[ \sum_{n=1}^\infty C_n n^{-s} e^{-n/Y} = \frac{1}{2 \pi i} \int_{2 - i \infty}^{2 + i \infty} L(s + w) M_X(s+w) Y^w \Gamma(w) dw. \]
We can write the left-hand side as
\[ e^{-1/Y} + \sum_{n > X} C_n n^{-s} e^{-n/Y}. \]
Further, supposing that $\sigma > 1/2$, we estimate the integral on the right-hand side.\footnote{The zero detection argument below is needed only for $\sigma>\frac{1}{2}$. The endpoint $\sigma=\frac{1}{2}$ follows immediately from the zero counting estimate in Lemma \ref{standardzerocount}.} We move the line of integration to $\Real(w) := 1/2 - \sigma$ and pick up a simple pole from $\Gamma$ at $w = 0$. If $L(s)$ has a pole at $1$, we also cross the corresponding pole of $L(s+w)$ at $w=1-s$ of order $-\ord_{s=1}L(s)$. Let $\mathcal{R}_L(s;Y)$ denote the resulting residue contribution at $w=1-s$, with the convention that $\mathcal{R}_L(s;Y) = 0$ if $L(s)$ is holomorphic at $1$. (If $L(s)$ has a simple pole at $1$, then $\mathcal{R}_L(s;Y)= \operatorname{Res}_{s=1} L(s) M_X(1) Y^{1-s}\Gamma(1-s)$.)
We can write the shifted integral as
\[ L(s) M_X(s) + \mathcal{R}_L(s;Y) + \frac{1}{2 \pi i} \int_{1/2 - \sigma - i \infty}^{1/2 - \sigma + i \infty} L(s + w) M_X(s+w) Y^w \Gamma(w)\,dw.\]
Given that $s = \rho = \beta + i \gamma$ is a zero of $L(s)$ and that $Y \geq 4$ so that $|e^{-1/Y} - 1| \leq \frac{1}{4}$, at least one of these expressions has an absolute value of at least $\frac{1}{4}$:
\begin{equation}\label{exps} \mathcal{R}_L(\rho;Y); \quad \sum_{n > X} C_n n^{-\rho} e^{-n/Y}; \quad \frac{1}{2 \pi i} \int_{1/2 - \beta - i \infty}^{1/2 - \beta + i \infty} L(\rho + w) M_X(\rho+w) Y^w \Gamma(w) \ dw. \end{equation}

Let $N_0, N_1, N_2$ denote the number of distinct $\rho = \beta + i\gamma$ with $\beta \geq \sigma$ and $|\gamma| \leq T$ such that each of the previous expressions, respectively, is at least $1/4$ in absolute value. We refer to $N_1$ and $N_2$ as the Type I and Type II contributions, respectively; we will show that $N_0$ is negligible. By the multiplicity bound in Lemma \ref{standardzerocount}, we have
\[ N_L(\sigma,T) \ll (N_0+N_1+N_2)\log T.\]

We bound $N_0$ in the following lemma.

\begin{lem}\label{n0} Suppose that $Y \leq T^6$ and $X = T^\eta$ with $0 < \eta < 1/2$. We have that
\[ N_0 \ll (\log T)^2. \]
\end{lem}

To bound $N_2$, we use the fact that $\rho+w$ lies on the critical line, allowing us to apply a critical-line second-moment estimate for $L(s)$.

\begin{prop}\label{n2}
Suppose that $Y\le T^6$ and $X=T^\eta$ with $0<\eta<1/2$. Then we have \[ N_2 \lessapprox X^2 Y^{2(1/2 - \sigma)} T \ll_\eta T^{1 + 3 \eta} Y^{1-2\sigma}.\]
\end{prop}

The Type I contribution is bounded using estimates for large values of Dirichlet polynomials. We can reduce the task of bounding $N_1$ to the task of bounding $|W|$ as follows.

\begin{prop}\label{n1} Suppose that $T\le Y\le T^6$, $X=T^\eta$ with $0<\eta<1/2$, and let $N_1$ denote the Type I count defined above. Then, there exists $C> 0$, an integer $N$ with
\[ \lfloor Y^{1/2} \rfloor \le N\le CY\log T,\]
coefficients $c_n$ supported on $n \sim N$ satisfying
\[ \sum_{n \sim N} |c_n|^q \ll_\eta N^{1+ o_\eta(1)}\] and a
$1$-separated set $W\subset [0,T]$ such that \[ N_1\lessapprox_{\eta} |W| \]
and such that the Dirichlet polynomial $D(t):=\sum_{n \sim N} c_n n^{it}$ satisfies, for all $t \in W$,
\[|D(t)|\geq N^{\sigma}.\]
\end{prop}

We will also make use of the following standard estimate for $|W|$.

\begin{prop}\label{original} Let $(c_n)_{n \in \NN} \subset \CC$ such that either $|c_n| \leq 1$ or, more generally, \mbox{$\sum_{n \sim N} |c_n|^2 \leq N$}. Suppose that $W$ is a set of $1$-separated points in $[0,T]$ such that 
\[ D(t) :=  \sum_{n \sim N} c_n n^{it}, \qquad |D(t)| \geq N^\sigma \]
for all $t \in W$. Then, we have Ingham's mean value theorem as in \cite{Ingham}
\begin{equation}\label{ingham}|W| \lessapprox N^{2 - 2 \sigma} + TN^{1 - 2 \sigma} \end{equation}
and Montgomery--Hal\'{a}sz--Huxley's large values estimate as in \cite{Huxley}
\begin{equation}\label{huxley} |W| \lessapprox N^{2 - 2 \sigma} + TN^{4 - 6 \sigma}. \end{equation}
\end{prop}

\section{The zero detection estimates}\label{n0n1n2}

We continue to regard $L$ and $q$ as fixed. Unless otherwise indicated, the implied constants may depend on $L$ and $q$ but are uniform in $T,\sigma,X,Y,N$, and in the zeros under consideration. The parameter $\eta\in(0,1/2)$ is always fixed independently of $T$, and dependence on $\eta$ is indicated explicitly. The parameters $\delta > 0$ and $A > 0$ denote auxiliary parameters and may take different values in different places; any relevant dependence on them is indicated where they occur.

The arguments in this section follow the classical zero detection method of Heath-Brown \cite{HB}. We include the details in order to verify the method under the hypotheses in Section~\ref{properties} and to keep track of the parameter dependence for later optimization. The estimate for $N_0$ does not use the coefficient-moment hypothesis, while the estimate for $N_2$ requires only the case $q=2$ of \eqref{qmoment}. The full $q$th-moment hypothesis enters only in the treatment of $N_1$. Apart from this additional bookkeeping, we follow standard methods.

\subsection{Proof of Lemma \ref{n0}}
Let $s = \sigma + it$. If $L(s)$ is holomorphic at $1$, then $\mathcal{R}_L(s;Y) = 0$ and hence, $N_0 = 0$. Otherwise, write $m := -\ord_{s=1} L(s) \geq 1$. Setting $z = w - (1-s)$, we have that 
\[ \mathcal{R}_L(s;Y) = \operatorname{Res}_{z = 0} L(1+z) M_X(1+z) Y^{1 - s + z} \Gamma(1-s+z).\]
For $s \neq 1$ with $\frac{1}{2} < \sigma \leq 1$, $M_X(s+w)$, $Y^{w}$, and $\Gamma(w)$ are holomorphic at $w = 1-s$, so the residue is a finite linear combination, with coefficients depending only on $L$, of terms of the form
\[ Y^{1-s}(\log Y)^u \, M_X^{(v)}(1)\, \Gamma^{(\ell)}(1-s), \qquad 0\le u,v,\ell\le m-1.\]
Hence, for $s \neq 1$ with $\frac{1}{2} < \sigma \leq 1$, we have the bound
\begin{equation}\label{Rbound}
|\mathcal{R}_L(s;Y)| \ \ll\ Y^{1-\sigma}(\log Y)^{m-1} \cdot \max_{0\le v\le m-1}|M_X^{(v)}(1)| \cdot \max_{0\le \ell\le m-1}|\Gamma^{(\ell)}(1-s)|.  \end{equation}
By Lemma \ref{bn-properties}, for any fixed $\delta>0$,
\[ |M_X^{(v)}(1)| \leq \sum_{n\leq X}\frac{|b_n|(\log n)^v}{n}\leq X^\delta(\log X)^v \sum_{n\geq1}\frac{|b_n|}{n^{1+\delta}} \ll_{L,\delta} X^\delta(\log X)^v. \]
Moreover, standard estimates for the derivatives of $\Gamma(s)$, obtained
from Stirling's formula, give
\[ |\Gamma^{(\ell)}(1-\sigma-it)| \ \ll_{\ell}\  (\log|t|)^{\ell}\,e^{-\pi/2|t|} \]
for $0\le 1-\sigma\le \frac{1}{2}$, and $|t| > 2$. (The zeros $\rho$ with $|\Ima(\rho)| \leq 2$ contribute $O(1)$.) Noting that $m \ll 1$ and substituting into \eqref{Rbound} with $s=\rho = \beta + i \gamma$ gives
\[ |\mathcal{R}_L(\rho;Y)| \ \ll\ Y^{1-\beta}\,(\log Y)^{O(1)}X^{1/2}(\log X)^{O(1)}\,(\log |\gamma|)^{O(1)} e^{-\pi/2|\gamma|} \ll T^{13/4+o(1)}e^{-\pi|\gamma|/2}, \]
since $\beta\geq\frac12$, $Y\leq T^6$, and $X=T^\eta$ with $\eta<1/2$. If $|\mathcal{R}_L(\rho;Y)|\ge \frac{1}{4}$, then $e^{-\pi/2|\gamma|}\gg T^{-13/4-o(1)}$, hence $|\gamma|\ll \log T$. By \eqref{zerocount} in Lemma \ref{standardzerocount} with $\log T$ substituted for $T$, we get $N_0\ll (\log T)\log\log T\ll (\log T)^2$, as required. \qed 

\subsection{Proof of Proposition \ref{n2}}

We first prove a second-moment bound for $L(s)$ on the critical line. We then show that each Type~II zero forces a nontrivial lower bound for
\[ \int_{\gamma-C\log T}^{\gamma+C\log T}\left|L\!\left(\frac{1}{2}+iu\right)\right|^2\,du. \]
Summing over all such zeros and using the local estimate $ n(u)\ll (\log T)^2$, where $n(u)$ counts the number of relevant zeros with $|\gamma-u|\le C\log T$, gives the required bound for $N_2$.

\begin{lem}\label{secondmoment} Let $L(s)$ be an admissible $L$-function of degree $2$ with $q$th-moment input for some $q \geq 2$. For $T\ge 2$,
\[ \int_{-T}^{T} \left|L\!\left(\frac{1}{2}+iu\right)\right|^2\,du \lessapprox T. \]
\end{lem}

\begin{proof}
It suffices to show that for every dyadic interval $[U,2U]$ with $U\ge 2$,
\begin{equation}\label{dyadic-second-moment} \int_U^{2U} \left|L\!\left(\frac{1}{2}+it\right)\right|^2\,dt \ll U^{1+o(1)}. \end{equation}
Indeed, the contribution from $|t|\le2$ is $O(1)$, and summing over the dyadic intervals  contained in $[2,T]$ and in $[-T,-2]$ gives the result. (The negative dyadic intervals are handled identically.)

Fix such a dyadic $U$. Choose an admissible auxiliary function $G$ which, together with its \mbox{derivatives}, decays rapidly on vertical lines. The approximate functional equation of \cite[Theorem 5.3]{IK} applies under the present hypotheses in view of Lemma \ref{polygrowth}, giving
\begin{equation}\label{IK-afe} L\!\left(\frac{1}{2}+it\right) = \mathcal{A}_1(t)+\varepsilon\!\left(\frac{1}{2}+it\right)\mathcal{A}_2(t)+\mathcal{R}(t), \end{equation}
where
\[ \mathcal{A}_1(t):=\sum_{n=1}^\infty \frac{a_n}{n^{1/2+it}} V_{\frac{1}{2}+it}\!\left(\frac{n}{\sqrt{Q}}\right), \qquad \mathcal{A}_2(t):=\sum_{n=1}^\infty \frac{\overline{a_n}}{n^{1/2-it}} V_{\frac{1}{2}-it}\!\left(\frac{n}{\sqrt{Q}}\right),\]
with $|\varepsilon(\frac{1}{2}+it)|=1$ and $V_s(y)$ an explicit smooth function corresponding to our choice of $G$ with integral formula \cite[(5.13)]{IK}. If $\Lambda(s)$ is entire, then $\mathcal{R}(t)=0$. Otherwise, $\mathcal{R}(t)$ is a finite sum of residue terms arising from the possible poles of $\Lambda(s)$ at $0$ and $1$. The rapid decay of $G$ and its derivatives on vertical lines, together with Stirling's formula, gives, for every $A>0$,
\[ \mathcal{R}(t)\ll_{A}(1+|t|)^{-A} \Rightarrow \int_U^{2U}|\mathcal{R}(t)|^2\,dt \ll_{A} \int_U^{2U}(1+t)^{-2A}\,dt \ll 1. \]
We now treat $\mathcal A_1$. For $r = 1, 2$, the coefficient $q$th moment hypothesis and H\"{o}lder's inequality give
\begin{equation}\label{lower-coefficient-moments} \sum_{n\leq x}|a_n|^r\ll x^{1+o(1)}. \end{equation}
Fix $\varepsilon > 0$. From the integral representation \cite[(5.13)]{IK} and Stirling's formula, for every $B>0$,
\begin{equation}\label{V-rapid-decay} V_{\frac{1}{2}+it}\left(\frac{n}{\sqrt Q}\right) \ll_B \left(1+\frac{n}{U}\right)^{-B}
\end{equation}
uniformly for $t\in[U,2U]$ and $n\geq1$. We claim that this permits the sum defining $\mathcal{A}_1(t)$ to be truncated at $n \leq U^{1+\varepsilon}$ with negligible error. By dyadic decomposition and \eqref{lower-coefficient-moments} with $r=1$,

\begin{align*} \sum_{n>U^{1+\varepsilon}} \frac{|a_n|}{\sqrt n} \left| V_{\frac{1}{2}+it}\left(\frac{n}{\sqrt Q}\right) \right|
&\ll_{B} U^B \sum_{\substack{N>U^{1+\varepsilon}\\N\ {\rm dyadic}}} N^{1/2-B+o(1)}\ll_{B} U^{(1+ \varepsilon)/2 - \varepsilon B + o(1)}. \end{align*}
Given $A>0$, choosing $B$ sufficiently large in terms of $A$ and $\varepsilon$ shows that this is $O_{A,\varepsilon}(U^{-A})$. Hence,
\[ \mathcal A_1(t) = \sum_{n\leq U^{1+\varepsilon}} \frac{a_n}{n^{1/2+it}} V_{\frac{1}{2}+it}\left(\frac{n}{\sqrt Q}\right) +O_{A,\varepsilon}(U^{-A}) \]
uniformly for $t \in [U, 2U]$. It remains to remove the $t$-dependence of the smooth weight before applying the mean-value theorem. Shifting the line of integration in \cite[(5.13)]{IK} to $\Real w=1/\log U$  expresses the truncated sum as an integral of ordinary Dirichlet polynomials in $t$, with coefficients independent of $t$. Stirling's formula and the rapid decay of $G$ show that the integral is absolutely convergent and that the $L^1$-norm of the corresponding Mellin weight is $U^{o(1)}$. Taking $L^2$-norms and applying the mean-value theorem for Dirichlet polynomials then yields
\[ \int_U^{2U}|\mathcal A_1(t)|^2\,dt \ll U^{1+\varepsilon+o(1)} \sum_{n\leq U^{1+\varepsilon}}\frac{|a_n|^2}{n}.\]
By \eqref{lower-coefficient-moments} with $r=2$ and dyadic summation,
\[ \sum_{n\leq  U^{1+\varepsilon}}\frac{|a_n|^2}{n} \ll U^{o(1)}. \]
Since $\varepsilon$ is arbitrary,
\[ \int_U^{2U}|\mathcal A_1(t)|^2\,dt \ll U^{1+o(1)}. \] The same argument applies to $\mathcal A_2$, and \eqref{IK-afe} therefore implies \eqref{dyadic-second-moment}. This proves the lemma.
\end{proof}

\begin{proof}[Proof of Proposition \ref{n2}]
Let
\[ I(\rho):=\frac{1}{2\pi i}\int_{1/2-\beta-i\infty}^{1/2 - \beta+i\infty} L(\rho + w) M_X(\rho+w) Y^w \Gamma(w)\,dw \]
denote the third quantity in \eqref{exps}, so that $N_2$ counts the distinct zeros $\rho=\beta+i\gamma$ with $\beta\ge \sigma$, $|\gamma|\le T$, and $|I(\rho)|\ge \frac{1}{4}$. We first dispose of the range $\frac{1}{2}\leq\sigma\leq\frac{1}{2}+\frac{1}{\log T}$. Since $Y\leq T^6$, we have $Y^{1-2\sigma}\geq Y^{-2/\log T}\geq e^{-12}$.
Therefore, by Lemma \ref{standardzerocount},
\[ N_2 \leq N_L(\sigma,T) \ll T\log T \lessapprox X^2Y^{1-2\sigma}T. \]
We may thus assume that $\sigma\geq\frac{1}{2}+\frac{1}{\log T}$. In particular, for every zero counted by $N_2$ we have $\frac{1}{2}-\beta \le -\frac{1}{\log T}$. Fix a zero $\rho$. On the line $w=1/2-\beta+iv$, we have
\[ I(\rho) = \frac{Y^{1/2-\beta}}{2\pi} \int_{-\infty}^{\infty} L\!\left(\frac{1}{2}+i(\gamma+v)\right) M_X\!\left(\frac{1}{2}+i(\gamma+v)\right) Y^{iv}\Gamma(\frac{1}{2}-\beta+iv)\,dv.\]
By \eqref{halfbound} and Lemma~\ref{bn-properties}, we have
\[ L\!\left(\frac12+iu\right)\ll(1+|u|)^{1/2}, \qquad \left|M_X\left(\frac{1}{2}+iu\right)\right| \leq X^{1/2+\delta} \sum_{n\geq1}\frac{|b_n|}{n^{1+\delta}} \ll_{\delta}X^{1/2+\delta}. \]
Choosing $\delta<\frac{1}{2}$, we obtain $M_X\left(\frac{1}{2}+iu\right)\ll X.$ Since $|\gamma|\le T$ and $X=T^\eta\ll T$, Stirling's gives
\begin{align*} &\int_{|v|>C\log T} \left|L\!\left(\frac{1}{2}+i(\gamma+v)\right)\right| \left|M_X\!\left(\frac{1}{2}+i(\gamma+v)\right)\right| \left|\Gamma(\frac{1}{2}-\beta+iv)\right|\,dv \\ &\qquad\ll T^{3/2}\int_{|v|>C\log T}e^{-\pi |v|/2}\,dv \ll T^{-100},
\end{align*}
uniformly for $T \geq 2$ provided $C$ is chosen sufficiently large. Hence, if $|I(\rho)|\ge \frac{1}{4}$,
\[ \frac{1}{8} \le \frac{Y^{1/2-\beta}}{2\pi} \int_{|v|\le C\log T} \left|L\!\left(\frac{1}{2}+i(\gamma+v)\right)\right| \left|M_X\!\left(\frac{1}{2}+i(\gamma+v)\right)\right| \left|\Gamma(\frac{1}{2}-\beta+iv)\right|\,dv. \]
For $|v|\le C\log T$, since $1/2-\beta\le -1/\log T$, the point $1/2-\beta+iv$ stays a distance $\gg 1/\log T$ from the pole of $\Gamma$ at $0$. As $\Gamma$ has no other poles in the strip $-\frac{1}{2}\le \Real z\le 0$, we have
\[ \Big\vert \Gamma(\frac{1}{2}-\beta+iv) \Big\vert \ll \frac{1}{|\frac{1}{2}-\beta+iv|}\ll \log T. \]
Using also the bound $|M_X(\frac{1}{2}+iu)|\ll X$, we deduce
\[ \frac{1}{8} \ll X Y^{1/2-\beta} (\log T) \int_{\gamma-C\log T}^{\gamma+C\log T} \left|L\!\left(\frac{1}{2}+iu\right)\right|\,du.\]
Since $1/2-\beta\le 1/2-\sigma \le 0$, this yields
\begin{equation}\label{L1-lower} \int_{\gamma-C\log T}^{\gamma+C\log T} \left|L\!\left(\frac{1}{2}+iu\right)\right|\,du \gg X^{-1}Y^{-(1/2-\sigma)}(\log T)^{-1}.
\end{equation}
By Cauchy--Schwarz,
\[ \left( \int_{\gamma-C\log T}^{\gamma+C\log T} \left|L\!\left(\frac{1}{2}+iu\right)\right|\,du \right)^2 \le (2C\log T) \int_{\gamma-C\log T}^{\gamma+C\log T} \left|L\!\left(\frac{1}{2}+iu\right)\right|^2\,du.\]
Combining this with \eqref{L1-lower}, we obtain
\begin{equation}\label{L2-lower} \int_{\gamma-C\log T}^{\gamma+C\log T} \left|L\!\left(\frac{1}{2}+iu\right)\right|^2\,du \gg X^{-2}Y^{-2(1/2-\sigma)}(\log T)^{-3}.
\end{equation}
Define $n(u)$ to be the number of $\rho$ contributing to $N_2$ with $|\gamma - u| \leq C \log T$. Summing \eqref{L2-lower} over $\rho$ contributing to $N_2$ gives
\[ N_2\,X^{-2}Y^{-2(1/2-\sigma)}(\log T)^{-3} \ll \int_{-T-C\log T}^{T+C\log T} \left|L\!\left(\frac{1}{2}+iu\right)\right|^2 n(u)\,du.\]
By \eqref{localzerocount} in Lemma \ref{standardzerocount}, we have $n(u)\ll (\log T)^2$ for all $|u|\le T+C\log T$, and therefore, we get
\[ N_2 \ll X^2Y^{1-2\sigma}(\log T)^5 \int_{-2T}^{2T} \left|L\!\left(\frac{1}{2}+iu\right)\right|^2\,du. \]
Combining with Lemma~\ref{secondmoment} proves the proposition.
\end{proof}

\subsection{Proof of Proposition \ref{n1}} 

The proof proceeds in two stages. Proposition \ref{n1intermediate} first extracts from the Type I sum in \eqref{exps} a single dyadic block
\[ \sum_{n \sim N} d_n n^{-\rho} \]
which is large for many zeros contributing to $N_1$. Here, Lemma \ref{typeI-tail} is used to truncate the original Type~I sum at $n \leq CY \log T$, while Lemma \ref{Cn-bound} provides the required $q$th-moment control.

We pass from large values at varying points $\rho=\beta+i\gamma$ to large values on the fixed vertical line $\Real s=\sigma$. By inserting a smooth weight and applying Fourier inversion, the displacement from $\beta$ to $\sigma$ is replaced by an average over nearby shifts in the ordinate, one of which must still give a large value on $\Real s= \sigma$. After a harmless renormalization of the coefficients, this produces a fixed Dirichlet polynomial
\[ D(t)=\sum_{n \sim N}c_n n^{it} \]
which is large on a $1$-separated set of ordinates.

Recall that the coefficients $C_n$ are defined by 
\begin{equation}\label{Cndef} L(s) M_X(s) = \sum_{n=1}^\infty C_n n^{-s} \ \Rightarrow \ C_n = \sum_{\substack{d \mid n \\ d \le X}} b_d a_{n/d}. \end{equation}

\begin{lem}\label{Cn-bound} The coefficients $C_n$ satisfy
\[ \sum_{n \sim N} |C_n|^q \ll N^{1 + o(1)}. \]
\end{lem}

\begin{proof}
By \eqref{Cndef} and H\"older's inequality, we have
\[ |C_n|^q \le \tau(n)^{q-1} \sum_{\substack{d \mid n \\ d \le X}} |b_d|^q |a_{n/d}|^q. \] Summing over $n \sim N$ and using $\tau(n)^{q-1}\ll_{\delta,q}n^\delta$, the moment bound \eqref{qmoment} for $a_n$, and Lemma \ref{bn-bound} for $b_n$, we obtain
\begin{align*} \sum_{n \sim N}|C_n|^q &\ll_{\delta,q} N^\delta \sum_{d\le N}|b_d|^q \sum_{N/d<m\le2N/d}|a_m|^q + N^\delta\sum_{d \sim N}|b_d|^q \\
&\ll_{\delta,q} N^{1+2\delta} \sum_{d\le N}\frac{|b_d|^q}{d^{1+\delta}} + N^\delta\sum_{d \sim N}|b_d|^q \\ &\ll_{\delta,q} N^{1+2\delta}. \end{align*}
Since $\delta > 0$ is arbitrary, the result for $C_n$ follows.
\end{proof}

\begin{lem}\label{typeI-tail}
Let $\frac{1}{2}\le \sigma\le 1$ and suppose that $Y\le T^6$. Then for every $A>0$, there exists a constant $C\geq 2$ such that
\[ \sum_{n>CY\log T}|C_n|n^{-\sigma}e^{-n/Y}\ll_{A} T^{-A}. \]
In particular, for $T \geq 2$,
\[ \sum_{n>CY\log T}|C_n|n^{-\sigma}e^{-n/Y}\le \frac{1}{8}. \]
\end{lem}

\begin{proof}
Writing $M:=CY\log T$, we decompose $n > M$ into dyadic intervals $(K, 2K]$. By H\"older's inequality and Lemma \ref{Cn-bound},
\begin{align*}
\sum_{n \sim K}|C_n|n^{-\sigma}e^{-n/Y} &\le K^{-\sigma}e^{-K/Y} \left(\sum_{n \sim K}|C_n|^q\right)^{1/q} K^{1-1/q} \\
&\ll K^{1-\sigma+o(1)}e^{-K/Y}.
\end{align*}
Since $\sigma\ge\frac{1}{2}$, this is $ \ll K^{1/2+o(1)}e^{-K/Y}$. Therefore, \[ \sum_{n>M}|C_n|n^{-\sigma}e^{-n/Y} \ll \sum_{j\ge0} (2^jM)^{1/2+o(1)} e^{-2^jM/Y} = \sum_{j\ge0} (2^jM)^{1/2+o(1)} T^{-C 2^j} \ll M^{1/2+o(1)} T^{-C}. \]
Since $Y\le T^6$, we have $M^{1/2+o(1)} \lessapprox_C T^{3}$ and, hence,
\[ \sum_{n>M}|C_n|n^{-\sigma}e^{-n/Y} \lessapprox_{C} T^{-C+3}. \]
Choosing $C$ sufficiently large in terms of $A$ proves the first claim. Enlarging $C$ further if necessary establishes the second.
\end{proof}

\begin{prop}\label{n1intermediate}
Suppose that $T\le Y\le T^6$, $X=T^\eta$ with $0<\eta<1/2$, and let $N_1$ denote the Type I count defined above. Then, there exist a constant $C \geq 2$, an integer $N$ with \[
\lfloor Y^{1/2}\rfloor\le N\le CY\log T, \] coefficients $d_n$, supported on $n \sim N$ with
\[ \sum_{n \sim N} |d_n|^q \ll_{\eta} N^{1 + o_\eta(1)} \]
and a set $\mathcal S$ of distinct  zeros $\rho=\beta+i\gamma$ of $L(s)$ with $\beta\ge \sigma,$ and $|\gamma|\le T$ whose ordinates are $1$-separated such that 
\[ N_1\ll_\eta |\mathcal S|(\log T)^2 \]
and
\[ \left|\sum_{n \sim N} d_n n^{-\rho}\right| \geq \frac{1}{\log T} \]
for all $\rho \in \mathcal{S}$.
\end{prop}

\begin{proof}
Let $\mathcal Z_1$ be the set of distinct zeros $\rho=\beta+i\gamma$ counted by $N_1$. From Lemma \ref{typeI-tail}, there is a constant $C \geq 2$ such that
\[ \left|\sum_{X<n\le CY\log T} C_n n^{-\rho}e^{-n/Y}\right|\ge \frac{1}{8} \]
for all $\rho \in \mathcal{Z}_1$. Following Jutila \cite{Jutila}, we first force the relevant range to begin at $Y^{1/2}$. Since $X = T^\eta < T^{1/2} \leq Y^{1/2}$, we let $m \geq 2$ be the largest integer such that $Y^{1/m}>X$ and define
\[ I_1=(Y^{1/2},CY\log T],\qquad I_r=(Y^{1/(r+1)},Y^{1/r}] \quad \forall \ 2\leq r\leq m-1, \qquad I_m=(X,Y^{1/m}]. \]
Since $X=T^\eta$ and $Y\leq T^6$, we have $m<6/\eta$ and hence $m = O_\eta(1)$. We write
\[ \Sigma_r(\rho):=\sum_{n\in I_r}C_n n^{-\rho}e^{-n/Y}, \qquad  \sum_{X<n\leq CY\log T}C_n n^{-\rho}e^{-n/Y} = \sum_{r=1}^m\Sigma_r(\rho). \]
Moreover, there exists $\ell(\rho)$ such that 
\[ |\Sigma_{\ell(\rho)}(\rho)| \geq \frac{1}{8m}.\] It follows that there exist $\ell\in\{1,\dots,m\}$ and a subset $\mathcal E\subseteq\mathcal Z_1$ such that $|\mathcal E|\gg_\eta N_1$ and $|\Sigma_{\ell}(\rho)| \gg_\eta 1$ for all $\rho \in \mathcal E$. Now, let $u_n:=C_ne^{-n/Y}\mathbbm 1_{I_\ell}(n)$ and define
\[ d_n:=\sum_{n_1\cdots n_\ell=n}u_{n_1}\cdots u_{n_\ell}, \qquad \Sigma_\ell(\rho)^\ell=\sum_n d_n n^{-\rho}. \]
The definition of $I_\ell$ shows that $d_n=0$ unless $Y^{1/2}<n\leq CY\log T$. Hence
\[ \left| \sum_{Y^{1/2}<n\leq CY\log T}d_n n^{-\rho} \right| = \vert \Sigma_\ell(\rho) \vert^\ell \gg_\eta1 \]
for all $\rho \in \mathcal E$.  Partition this range into $J\ll\log T$ dyadic intervals. A second application of the pigeonhole principle gives an integer $N$ with $\lfloor Y^{1/2}\rfloor\leq N\leq CY\log T$ and a subset $\mathcal E'\subseteq\mathcal E$ such that
\[ |\mathcal E'| \gg \frac{|\mathcal E|}{\log T} \gg_\eta\frac{N_1}{\log T}, \qquad \left|\sum_{n\sim N}d_n n^{-\rho}\right| \gg_\eta\frac{1}{\log T} \]
for all $\rho \in \mathcal E'$.
It remains to control the coefficients $d_n$. Set $v_n = |u_n|^q$. By Lemma \ref{Cn-bound} and dyadic summation,
\[ \sum_{n\le x}|u_n|^q = \sum_{n \le x} v_n \ll x^{1+o(1)}. \]
Hence, by H\"older's inequality,
\[ |d_n|^q \le \tau_\ell(n)^{q-1} \sum_{n_1\cdots n_\ell=n} |u_{n_1}\cdots u_{n_\ell}|^q \ll_{\eta} n^{o_\eta(1)} (v^{(*\ell)})(n), \]
where $v^{(*\ell)}$ denotes the $\ell$-fold Dirichlet convolution of $v$. Since $\ell=O_\eta(1)$, the standard fixed-convolution estimate gives
\[ \sum_{n\le x}(v^{(*\ell)})(n) \ll_{\eta}x^{1+o_\eta(1)}.\]
Therefore,
\[ \sum_{n \sim N}|d_n|^q \ll_{\eta}N^{1+o_\eta(1)}.\]
Restricting $d_n$ to $n\sim N$ and setting $d_n = 0$ otherwise and rescaling by a constant depending only on $\eta$, we may assume that
\[ \left|\sum_{n\sim N}d_n n^{-\rho}\right| \geq\frac{1}{\log T} \]
for all $\rho \in \mathcal E'$ while retaining
\[ \sum_{n\sim N}|d_n|^q\ll_\eta N^{1+o_\eta(1)}. \]
Finally, choose a maximal subset $\mathcal S\subseteq \mathcal E'$ such that $|\Ima\rho-\Ima\rho'|\ge 1$ for all $\rho\neq \rho'$ in $\mathcal S$. By \eqref{localzerocount} in Lemma \ref{standardzerocount}, each interval of length $2$ contains $O(\log T)$ zeros, so
\[ |\mathcal E'|\ll |\mathcal S|\log T. \]
Since $|\mathcal E'|\gg_\eta N_1/\log T$, we obtain
\[ N_1\ll_\eta |\mathcal S|(\log T)^2.\]
\end{proof}

\begin{proof}[Proof of Proposition \ref{n1}]

It remains to transfer the large values from the zeros  $\rho=\beta+i\gamma$ to the fixed line $\Real s=\sigma$. Let $N$, $(d_n)$, and $\mathcal S$ be given by Proposition \ref{n1intermediate}. Write
\[ P(s):=\sum_{n \sim N} d_n n^{-s} \]
so that
\[ N_1\ll_\eta |\mathcal S|(\log T)^2, \qquad |P(\rho)|\geq \frac{1}{\log T} \]
for all $\rho \in \mathcal S$. Choose a smooth function $w_0 \in C_c^\infty((0, \infty))$ supported on $[1/2,4]$ and equal to $1$ on $[1,2]$. For $\rho=\beta+i\gamma\in\mathcal S$, define
\[ \psi_\rho(u):=e^{-(\beta-\sigma)(u-\log N)}\,w_0(e^u/N). \]
Since $\beta\in[\sigma,1]$, we have $\beta-\sigma\in[0,1]$; it follows by repeated differentiation that for each $j\ge 0$ one has $\|\psi_\rho^{(j)}\|_\infty \ll_j 1$ uniformly in $\rho$. On the interval $u\in[\log N,\log 2N]$ we have
\[ \psi_\rho(u)=e^{-(\beta-\sigma)(u-\log N)}. \]
Hence, for $\rho=\beta+i\gamma\in\mathcal S$,
\[ P(\rho) = \sum_{n \sim N} d_n n^{-\beta-i\gamma} = N^{\sigma-\beta} \sum_{n \sim N} d_n n^{-\sigma-i\gamma}\psi_\rho(\log n). \]
By Fourier inversion,
\[ \sum_{n \sim N} d_n n^{-\sigma-i\gamma}\psi_\rho(\log n) = \int_{\RR}\widehat{\psi_\rho}(\xi)\, P\bigl(\sigma+i(\gamma-2\pi\xi)\bigr)\,d\xi. \]
The Fourier transforms $\widehat{\psi_\rho}$ decay rapidly, uniformly in $\rho$. Moreover, by H\"older's inequality and the coefficient moment estimate,
\begin{align*}
|P(\sigma+iu)| &\leq N^{-\sigma} \left(\sum_{n \sim N}|d_n|^q\right)^{1/q} N^{1-1/q} \ll_{\eta} N^{1-\sigma + o_\eta(1)} \ll_\eta T^{O(1)}
\end{align*}
uniformly in $u$. Fix $0 < \delta < 1$. The rapid decay of $\widehat{\psi_\rho}$, together with the preceding polynomial bound for $P(\sigma + iu)$, allows the Fourier integral to be truncated at $|\xi|\leq T^\delta$ with error $O_{\eta, \delta}(T^{-100})$. Hence,
\[ \left| \int_{|\xi|\leq T^\delta} \widehat{\psi_\rho}(\xi)\, P\bigl(\sigma+i(\gamma-2\pi\xi)\bigr)\,d\xi \right| \gg \frac{N^{\beta-\sigma}}{\log T} \]
for all sufficiently large $T$. Since $\|\widehat{\psi_\rho}\|_{L^1(\RR)}\ll1$ uniformly in $\rho$, there is an absolute constant $c_0>0$ such that, for each $\rho\in\mathcal S$, one may choose $|\xi_\rho|\leq T^\delta$ with
\[ \left| P\bigl(\sigma+i(\gamma-2\pi\xi_\rho)\bigr) \right| \geq c_0\frac{N^{\beta-\sigma}}{\log T}. \]
Set $t_\rho:=-(\gamma-2\pi\xi_\rho)$. Since $|\xi_\rho|\leq T^\delta$ and $|\gamma|\leq T$, we have $|t_\rho|\leq2T$ for all large $T$. Define
\[ c_n^{(0)}:=d_n\Big(\frac{N}{n}\Big)^\sigma \]
for $n\sim N$ and
\[ D^*(t):=\sum_{n\sim N}c_n^{(0)}n^{it} =N^\sigma P(\sigma-it). \]
Since $n\sim N$,
\[ \sum_{n\sim N}|c_n^{(0)}|^q \leq \sum_{n\sim N}|d_n|^q \ll_\eta N^{1+o_\eta(1)}. \]
Moreover, for every $\rho\in\mathcal S$,
\[ |D^*(t_\rho)| = N^\sigma \left|P\bigl(\sigma+i(\gamma-2\pi\xi_\rho)\bigr)\right| \geq c_0\frac{N^\beta}{\log T} \geq c_0\frac{N^\sigma}{\log T}. \]
Set $c_n:=c_0^{-1}c_n^{(0)}$ and
\[ D(t):=\sum_{n\sim N}c_n n^{it}, \qquad \sum_{n\sim N}|c_n|^q = c_0^{-q}\sum_{n \sim N}|c_n^{(0)}|^q \ll_{\eta}N^{1+o_\eta(1)}. \]
Writing $\mathcal T:=\{t_\rho:\rho\in\mathcal S\} \subseteq[-2T,2T]$, we have
\[ |D(t)|\geq\frac{N^\sigma}{\log T} \]
for all $t \in \mathcal{T}$. Choose a maximal $1$-separated subset $W_0\subseteq\mathcal T$. Since $|t_\rho+\gamma|\ll T^\delta$ and the ordinates in $\mathcal S$ are $1$-separated, each point of $W_0$ corresponds to at most $O(T^\delta)$ elements of $\mathcal S$. Hence, $|\mathcal S|\ll T^\delta|W_0|$
and therefore
\begin{equation}\label{prevestimate} N_1\ll_{\eta,\delta}T^\delta|W_0|(\log T)^2. \end{equation}
Partition $W_0$ according to the four intervals $[-2T,-T]$, $[-T,0]$, $[0,T]$, and $[T,2T]$, and keep the part $W$ of maximal cardinality. Thus, $|W_0|\ll|W|$. If $W$ lies in a negative interval, replace $W$ by $-W$ and $c_n$ by $\overline{c_n}$; if the resulting set lies in $[T,2T]$, replace $W$ by $W-T$ and $c_n$ by $c_n n^{iT}$. These operations preserve the coefficient magnitudes and the large-values condition. Hence, we may assume that $W\subseteq[0,T]$. Since $\delta > 0$ is arbitrary, \eqref{prevestimate} gives \[ N_1 \lessapprox_\eta |W|. \]
Finally, since $Y\geq T$ and $N\geq\lfloor Y^{1/2}\rfloor\gg T^{1/2}$, we may replace $c_n$ by $(\log T)c_n$ while retaining the $q$th moment bound. After relabeling the coefficients, we have $|D(t)|\geq N^\sigma$ for all $t \in W$. \end{proof}

\section{Proof of Theorem \ref{momentGM}}\label{mainprop}

We follow the argument of Guth and Maynard \cite{GM}, keeping track of the dependence on the size of the coefficients after decomposing them into dyadic size classes. Let $w: \RR \to [0,1]$ be a fixed smooth function supported on $[1,2]$ with $\Vert w^{(j)} \Vert_\infty = O_j(1)$ for all $j \in \ZZ_{\geq 0}$ and with $w(t) = 1$ for $t \in [6/5, 9/5]$.
We first reduce Theorem \ref{momentGM} to a local estimate at the scale $T = N^{6/5}$, inserting the smooth weight $w$ in the $n$ variable, following \cite[Section~3]{GM}. This is the natural local scale at which the two final terms in Theorem \ref{momentGM} balance.

\begin{prop}\label{momentGM-local}
Let $q\geq 8/3$, $\varepsilon > 0$, and $7/10 \leq \sigma \leq 4/5$. Suppose that $(c_n)_{n \sim N} \subset \CC$ satisfies
\[ \sum_{n\sim N}|c_n|^q\ll N^{1+o(1)}. \]
Let $W$ be a $T^\varepsilon$-separated set contained in an interval of length $T = N^{6/5}$, and suppose that
\[ \left| \sum_{n\geq 1} w\left(\frac{n}{N}\right)c_n n^{it} \right| \geq N^\sigma \]
for every $t \in W$. Then, we have
\[ |W| \lessapprox_\varepsilon  TN^{(12-20\sigma)/5}. \]
\end{prop}

\begin{proof}[Proof of Theorem \ref{momentGM} assuming Proposition \ref{momentGM-local}]

By H\"older's inequality and the coefficient moment hypothesis, we have
\[ \sum_{n\sim N}|c_n|^2 \ll N^{1+o(1)}. \]
If $N\ge T$, the mean-value estimate for Dirichlet polynomials gives $|W|\ll N^{2+o(1)}V^{-2}$, which is sufficient for Theorem \ref{momentGM}. We may therefore suppose that $N < T$. In this range, all $N^{o(1)}$ losses may be absorbed into $\lessapprox$. The classical mean-value and Montgomery--Hal\'{a}sz--Huxley estimates therefore give Theorem \ref{momentGM} if $V\leq N^{7/10+o(1)}$  or $V\geq N^{4/5-o(1)}$. We may therefore suppose that $4N^{7/10}\leq V\leq N^{4/5}$, in which case we have that $N^2 V^{-2} \leq N^{18/5} V^{-4}$. 

We now insert the smooth weight. Split $D(t)$ as $D_1(t) + D_2(t)$ where
\[ D_1(t) := \sum_{N < n \leq 3N/2} c_n n^{it} \qquad \text{and} \qquad D_2(t) := \sum_{3N/2 < n \leq 2N} c_n n^{it}. \]
If $|D(t)| \geq V$, then either $|D_1(t)| \geq V/2$ or $|D_2(t)| \geq V/2$. It therefore suffices to bound the corresponding sets of large values separately. For the first sum, let $N^{(1)} = 5N/6$, so that $[\frac{N}{N^{(1)}}, \frac{3N}{2N^{(1)}}] = [\frac{6}{5}, \frac{9}{5}]$. Similarly, let $N^{(2)} = 5N/4$ so that $[\frac{3N}{2N^{(2)}}, \frac{2N}{N^{(2)}}] = [\frac{6}{5}, \frac{8}{5}] \subset [\frac{6}{5}, \frac{9}{5}]$.

Let $c_n^{(j)}$ denote $c_n$ restricted to the support of $D_j$. Since $w(u)=1$ on $[6/5,9/5]$, we have
\[ D_j(t)= \sum_{n\geq1} w\left(\frac{n}{N^{(j)}}\right)c_n^{(j)}n^{it}. \]
Partition $W=W_1\cup W_2$ so that $|D_j(t)|\geq V/2$ for $t\in W_j$, and define
\[ \widetilde c_n^{(j)} := 2\left(\frac{N^{(j)}}{N}\right)^\sigma c_n^{(j)}. \]
\vspace{-2em}

\noindent Then, for $t\in W_j$, we have
\[ \left| \sum_{n \geq 1} w\left(\frac{n}{N^{(j)}}\right) \widetilde c_n^{(j)}n^{it} \right| \geq \left( N^{(j)} \right)^\sigma \ \text{with} \  \sum_{n \geq 1} |\widetilde c_n^{(j)}|^q \ll N^{1+o(1)} \ll \left( N^{(j)} \right)^{1+o(1)}\]
since $N^{(j)}\asymp N$. It therefore suffices to bound each $W_j$ separately. Fixing $j\in\{1,2\}$, we now relabel $N^{(j)}\ \text{as }N, \ \widetilde c_n^{(j)}\ \text{as }c_n,\ \text{and} \ W_j\ \text{as }W.$ Fix $\delta>0$, and choose the separation parameter $\varepsilon>0$ in Proposition \ref{momentGM-local} sufficiently small in terms of $\delta$. There exists a $T^\delta$-separated subset $W'\subseteq W$ such that $|W|\ll T^\delta|W'|$. Suppose that $T\leq N^{6/5}$. We regard $W'$ as contained in an interval of length $N^{6/5}$. Then, Proposition \ref{momentGM-local} gives
\[ |W| \ll T^\delta|W'| \lessapprox_\delta T^\delta N^{18/5-4\sigma}. \]
Now, suppose instead that $T>N^{6/5}$. Partition the interval into $O(T/N^{6/5})$ intervals of length at most $N^{6/5}$, and write $W'_j$ for the points of $W'$ in the $j$-th interval. Applying Proposition \ref{momentGM-local} to each $W'_j$ gives
\[ |W| \ll T^\delta\sum_j|W'_j| \lessapprox_\delta T^{1+\delta}N^{12/5-4\sigma}. \]
Since $\delta$ is arbitrary, this proves the theorem in either case.
\end{proof}

For the remainder of the section, we work under the hypotheses of Proposition \ref{momentGM-local}; in particular, $T = N^{6/5}$. We also suppress dependence on the fixed separation parameter $\varepsilon$ in the notation $\lessapprox$. 

We now perform a dyadic decomposition on the coefficients $c_n$ by size. We separate the coefficients of size at most $1$ by defining
\[ \mathcal I_0:=\{n\sim N:|c_n|\leq 1\}, \qquad D_0(t):= \sum_{n\in\mathcal I_0} w\left(\frac{n}{N}\right)c_n n^{it}. \]
For dyadic $B\geq1$, let
\[ \mathcal I_B:=\{n\sim N:B<|c_n|\leq2B\}, \qquad D_B(t):= \sum_{n\in\mathcal I_B} w\left(\frac{n}{N}\right)c_n n^{it}. \]
Let $\mathbf{c}_B = (c_n \mathbbm{1}_{\mathcal{I}_B}(n))_{n \sim N}$ be the coefficient vector supported on $\mathcal{I}_B$. From the coefficient $q$th-moment hypothesis, we have that 
\begin{equation}\label{IBbound} |\mathcal{I}_B| \ll N^{1+o(1)} B^{-q}, \end{equation}
and hence
\begin{equation}\label{cBbound} \|\mathbf c_B\|_2^2 = \sum_{n\in\mathcal I_B}|c_n|^2 \ll N^{1+o(1)}B^{2-q}.
\end{equation}

Let $\mathcal B$ denote the set of dyadic $B\geq1$ for which $\mathcal I_B\neq\varnothing$, and put $J_B:=1+|\mathcal B|$. Since
\[ \max_{n\sim N}|c_n| \leq \left(\sum_{n\sim N}|c_n|^q\right)^{1/q} \leq N^{1/q+o(1)}, \]
we have $J_B\ll\log N$. For every $t\in W$, the inequality $|D(t)|\geq N^\sigma$ implies that $|D_\nu(t)|\geq\frac{N^\sigma}{J_B}$ for at least one $\nu\in\{0\}\cup\mathcal B$. For $\nu\in\{0\}\cup\mathcal B$, define
\[ W_\nu:= \left\{t\in W: |D_\nu(t)|\geq\frac{N^\sigma} {J_B} \right\}. \]
Then $W=\bigcup_{\nu\in\{0\}\cup\mathcal B}W_\nu$, so there exists $\nu$ such that $|W|\leq J_B|W_\nu|$. If $\nu=0$, the coefficients of $D_0(t)$ are $1$-bounded, and the required estimate follows directly from the large-values theorem of Guth and Maynard, applied with threshold $N^\sigma/J_B$. Since $J_B\ll\log N$, all resulting powers of $J_B$ are absorbed by $\lessapprox$. We may therefore suppose that $\nu=B\in\mathcal B$.

Fix this dyadic block and replace $W$ by $W_B$, which we will relabel as $W$ to simplify notation. Multiplying the coefficients of $D_B$ and the parameter $B$ by $J_B$, and then relabelling, we may assume that
\[ B<|c_n|\leq2B\quad \forall \ n\in\mathcal I_B, \qquad |D_B(t)|\geq N^\sigma\quad \forall \ t\in W. \]
The support $\mathcal I_B$ is unchanged. Since $J_B=N^{o(1)}$, the rescaled coefficients continue to satisfy
\[ \sum_{n\sim N}|c_n|^q\ll N^{1+o(1)}, \]
and the bounds \eqref{IBbound} and \eqref{cBbound} remain valid. Moreover, the dyadic pigeonholing loses only a factor $J_B\ll\log N$, which is absorbed by $\lessapprox$.

\begin{lem}\label{initial}
Let $B$, $D_B$, and $W$ be fixed as above, let $M$ be the $|W| \times N$ matrix with entries
\[ M_{t,n} = w\left(\frac{n}{N}\right)n^{it}, \qquad t\in W,\quad n\sim N \]
and $s_1(M)$ be its largest singular value. Then, we have
\[ |W| \lessapprox N^{1-2\sigma}B^{2-q}s_1(M)^2. \]
\end{lem}

\begin{proof}
Since $D_B(t)=(M\mathbf c_B)_t$ and $|D_B(t)|\geq N^\sigma$ for every $t\in W$, we have
\[ |W|N^{2\sigma} \leq \sum_{t \in W} |D_B(t)|^2 = \|M\mathbf c_B\|_2^2 \leq s_1(M)^2\|\mathbf c_B\|_2^2. \]
The bound \eqref{cBbound} now gives
\[ |W|N^{2\sigma} \lessapprox NB^{2-q}s_1(M)^2, \]
which proves the result.
\end{proof}

As $s_1(M)^2$ is the largest eigenvalue of $MM^*$, we can bound it, for $k \in \NN$, using the inequality $s_1(M)^{2k} \leq \tr ((MM^*)^k)$. In \cite{GM}, Guth and Maynard make two observations. First, sufficiently sharp bounds for $\tr((MM^*)^k)$ appear inaccessible for large $k$, so they work with $k = 3$. Second, to produce a more precise bound, note that by H\"{o}lder's inequality,
\[ \tr((MM^*)^3) \geq \frac{(\tr(MM^*))^3}{|W|^2}, \]
with equality when the eigenvalues of $MM^*$ are all equal. Thus, if
\[ \tr((MM^*)^3) - \frac{(\tr(MM^*))^3}{|W|^2} \]
is small, we can expect most of the contribution to the trace to be coming from other terms, resulting in a tighter bound on the largest singular value. More precisely, \cite[Lemma~4.2]{GM} gives
\begin{equation}\label{s1M} s_1(M) \leq 2\left( \tr((MM^*)^3) - \frac{(\tr(MM^*))^3}{|W|^2} \right)^{1/6} + 2\left(\frac{\tr(MM^*)}{|W|}\right)^{1/2}.
\end{equation}
Now, put $h_t(u):=w(u)^2u^{it}$. The trace calculations of \cite[Lemmas~4.4 and~4.5]{GM}, which depend only on the matrix $M$ and therefore apply to our argument without modification, give
\begin{equation}\label{firsttrace} \tr(MM^*) = N|W|\|w\|_{L^2}^2+O_A(N^{-A}) \end{equation}
and
\begin{equation}\label{thirdtrace} \tr((MM^*)^3) = N^3|W|\|w\|_{L^2}^6 + \sum_{\mathbf m\in\ZZ^3\setminus\{\mathbf0\}}I_{\mathbf m} + O_{A,\varepsilon}(T^{-A}),
\end{equation}
where
\begin{equation}\label{Imdef} I_{\mathbf m} := N^3 \sum_{t_1,t_2,t_3\in W} \widehat h_{t_1-t_2}(m_1N) \widehat h_{t_2-t_3}(m_2N) \widehat h_{t_3-t_1}(m_3N). \end{equation}
Combining these estimates with Lemma \ref{initial} gives the following
proposition.

\begin{prop}\label{keyWestimate} Let $W$ be $T^\varepsilon$-separated, and let $D_B(t)$ be such that the coefficients $c_n$ satisfy $B < |c_n| \leq 2B$ and $|D_B(t)| \geq N^{\sigma}$ for all $t \in W$. Then, we have
\[ |W| \lessapprox B^{2-q} \left( N^{2 - 2 \sigma} + N^{1-2 \sigma} \left( \sum_{\mathbf m \in \ZZ^3 \setminus \{ \mathbf{0} \} } I_{\mathbf m} \right)^{1/3} \right). \]
\end{prop}

Following the notation of \cite{GM}, we divide the sum into pieces
\[ \sum_{\mathbf{m} = (m_1, m_2, m_3) \in \ZZ^3 \setminus \{ \mathbf{0} \}} I_{\mathbf{m}} = S_1 + S_2 + S_3 \]
where for $j \in \{ 1, 2, 3 \}$, $S_j$ contains the terms where exactly $j$ of the $m_i$ are non-zero.
We now identify where the modified coefficient hypothesis affects the argument of \cite{GM}. Although the set $W$, and hence the quantities $S_1, S_2, S_3$, arise from the large-value condition for $D_B(t)$, the proofs of the bounds for $S_1$ and $S_2$ use only the spacing property of $W$ and make no further use of the coefficients. These parts of the argument therefore carry over unchanged. The following bound shows that $S_1$ is negligible.

\begin{prop}[{\cite[Proposition 5.1]{GM}}]\label{s1} For any $A > 0$, we have 
\[ S_1 = O_{\varepsilon,A}(T^{-A}). \]
\end{prop}

The term $S_2$ is controlled by Heath-Brown's large-values estimate on difference sets \cite{HB2}. Indeed, the cubic trace collapses to an expression of the form
\[ \sum_{t_1,t_2\in W} \left| \sum_{m\neq 0}\widehat h_{t_1-t_2}(mN) \right|^2. \]
After decomposing according to the size of $|t_1-t_2|$, the reflection principle converts the inner Fourier sum into a Dirichlet polynomial of dual length comparable to $|t_1-t_2|/N$. Guth and Maynard introduce a flexible parameter $k \in \NN$ by applying H\"older's inequality and rewrite a higher moment of this polynomial as the second moment of a longer Dirichlet polynomial. This places the expression precisely in the setting of Heath-Brown's estimate, producing Proposition \ref{s2}.

\begin{prop}[{\cite[Proposition 6.1]{GM}}]\label{s2} For any $k \in \NN$, we have 
\[ S_2 \lessapprox_{k} N^2 |W|^2 + TN|W|^{2-1/k} + N^2 |W|^2 \left( \frac{T^{1/2}}{|W|^{3/4}} \right)^{1/k}. \]
\end{prop}

$S_3$ contains the main harmonic-analytic input of \cite{GM}. We use the large-value condition for $D_B$ to bound the additive energy of $W$; we first explain how this quantity enters the estimate for $S_3$.

Recall that the quantities $I_{\mathbf m}$ involve sums over $t_1, t_2, t_3 \in W$. A natural approach after Poisson summation would be to apply stationary phase separately to each of the three Fourier integrals. Although this gives strong pointwise estimates for the individual integrals, the resulting phases remain coupled through the differences $t_1 - t_2$, $t_2 - t_3$, and $t_3 - t_1$.

One key observation in \cite{GM} is that, to allow for additional cancellation, it is actually advantageous not to simplify the Fourier integrals at this stage. Expanding the Fourier transforms in $I_{\mathbf m}$ and summing first over $t_1, t_2, t_3$, Guth and Maynard introduce the function $R(x) := \sum_{t \in W} x^{it}$ and reduce the problem to estimating oscillatory integrals containing
\[ R\left(\frac{u_1}{u_3}\right) R\left(\frac{u_2}{u_1}\right) R\left(\frac{u_3}{u_2}\right). \]
The three arguments are not independent: their product is $1$. With the change of variables $v_1 = u_1/u_3, v_2 = u_2/u_3$, repeated integration by parts shows the contribution to the integral is negligible unless $m_1 v_1 + m_2 v_2 + m_3$ is close to $0$. Thus, the arguments of $R$ satisfy an approximate affine relation, creating an opportunity for additional cancellation that could compensate for the savings sacrificed at the stationary-phase step. The problem becomes controlling how often $R$ can be simultaneously large at such triples of points.

We first obtain a preliminary bound using moment estimates for $R$. This is where additive energy enters the argument; it governs the average size of $R$. Guth and Maynard establish
\[ \int_{v \asymp 1} |R(v)|^2 \, dv \lessapprox |W|, \qquad \int_{v \asymp 1} |R(v)|^4 \, dv \lessapprox E(W), \]
in \cite[Lemmas 8.2 and 8.3]{GM}, where $E(W)$ denotes \textit{additive energy}, defined by
\[ E(W) = \# \{ w_1, w_2, w_3, w_4 \in W: |w_1 + w_2 - w_3 - w_4| \leq 1 \}.\]
Applying H\"older's inequality gives the preliminary estimate
\begin{equation}\label{s3bound}
S_3 \lessapprox T^2|W|^{1/2}E(W)^{1/2}.
\end{equation}
When $E(W)$ is close to its minimal scale $|W|^2$, this already gives the tight bound $S_3\lessapprox T^2|W|^{3/2}$. This handles the case of \textit{low additive energy}.

At the opposite extreme, \textit{high additive energy} means that $W$ contains many additive relations; as we will see, this is precisely the type of structure detected by Heath-Brown's difference-set estimate. The main difficulty lies between these two regimes in the case of \textit{intermediate additive energy}.

To bound $S_3$ more closely in terms of $E(W)$, Guth and Maynard analyze when \eqref{s3bound} could be close to sharp. Exploiting averaging over $m_1, m_2, m_3$, they show that large values of $R$ cannot repeatedly align under these affine transformations. The resulting extra cancellation gives the following proposition, which gains a factor of $N/T$ when $E(W)>|W|^2$.

\begin{prop}[{\cite[Proposition 10.1]{GM}}]\label{S3} If $W$ is a $T^\varepsilon$-separated set contained in an interval of length $T$, we have
\[ S_3 \lessapprox T^2 |W|^{3/2} + TN|W|^{1/2} E(W)^{1/2}. \]
\end{prop}

We do not modify this part of the argument. To handle sets $W$ with high additive energy, it remains to directly bound $E(W)$, and the large-value hypothesis for the underlying Dirichlet polynomial re-enters at this point. On the fixed dyadic block $\mathcal I_B$, set 
\[ \widetilde c_n := \frac{c_n}{2B} \mathbbm 1_{\mathcal I_B}(n), \qquad \widetilde D_B(t) := \frac{1}{2B} D_B(t)\]
so that $|\widetilde c_n| \leq 1$ and $|\widetilde D_B(t)| \geq \frac{N^{\sigma}}{2B}$ for all $t \in W$. The normalization leaves the set $W$, and hence $E(W)$, unchanged. Thus, applying \cite[Lemma~11.4]{GM} to $\widetilde D_B(t)$ gives 
\begin{equation*} E(W) \lessapprox B^2 N^{-2 \sigma} \sum_{n_1, n_2 \in \mathcal{I}_B} \Big\vert R \left( \frac{n_1}{n_2} \right) \Big\vert^3. \end{equation*}
For the present argument, we do not attempt to exploit the support of $\mathcal I_B$ further (see Remark \ref{longrmk}) and we use
\begin{equation}\label{newAEbound} E(W) \lessapprox B^2 N^{-2 \sigma} \sum_{n_1, n_2 \sim N} \Big\vert R \left( \frac{n_1}{n_2} \right) \Big\vert^3. \end{equation}

In the proof of \cite[Lemma~11.4]{GM}, the local constancy of the Dirichlet polynomial allows an approximate relation $t_1+t_2-t_3\approx t_4$ to be tested using $|\widetilde D_B(t_1+t_2-t_3)|^2$. Expanding this square and summing over $t_1,t_2,t_3\in W$ produces three copies of $R$, which combine to give $\left| R\left(\frac{n_1}{n_2}\right) \right|^3$.

The remaining problem is therefore to estimate this discrete cubic moment. The difference-set estimate of Heath-Brown \cite{HB2} plays a key role in bounding the discrete second and fourth moments; the use of this estimate requires $\ell^\infty$-bounded coefficients (see \cite[Lemmas 11.5 and 11.6]{GM}).

A direct interpolation between the discrete second- and fourth-moment estimates is, however, inefficient since the two moments are governed by different ranges of $\gcd(n_1, n_2)$. Guth and Maynard therefore split the sum at $\Delta = \frac{N^2}{T}$ and treat the small- and large-gcd ranges separately. 

\begin{prop}\label{EWprop} Let $D_B(t)$ and $W$ be as in Proposition \ref{keyWestimate}. Then, we have 
\[ E(W) \lessapprox |W| B^4 N^{4 - 4 \sigma} + B^2 T^{1/4} |W|^{21/8} N^{1-2\sigma} + B^2 |W|^3 N^{1-2\sigma}. \]
\end{prop}

\begin{proof}
Starting from \eqref{newAEbound}, put $\Delta := \frac{N^2}{T}$ and split the discrete cubic moment according to whether $\gcd(n_1, n_2) \leq \Delta$ or $\gcd(n_1, n_2) \geq \Delta$. For the small-gcd contribution, \cite[Lemma~11.8]{GM} gives
\begin{equation}\label{smallgcd} \sum_{\substack{n_1, n_2 \sim N \\ \gcd(n_1, n_2) \leq \Delta}} \Big\vert R \left( \frac{n_1}{n_2} \right) \Big\vert^3 \lessapprox (\Delta T + N^2)|W|^{1/2} E(W)^{1/2} \lessapprox N^2|W|^{1/2} E(W)^{1/2}.\end{equation}
For the large-gcd contribution, \cite[Lemma 11.9]{GM} gives
\begin{equation}\label{largegcd} \sum_{\substack{n_1, n_2 \sim N \\ \gcd(n_1, n_2) \geq N^2/T}} \Big\vert R \left( \frac{n_1}{n_2} \right) \Big\vert^3 \lessapprox N|W|^3 + NT^{1/4}|W|^{21/8} + E(W)^{1/2} |W|^{1/2} N^2. \end{equation}
Here, with our choice $T = N^{6/5}$, the hypothesis $N \geq T^{3/4}$ is satisfied.
Combining \eqref{smallgcd} and \eqref{largegcd} with \eqref{newAEbound} gives
\[ E(W) \lessapprox B^2N^{-2\sigma} \left( N|W|^3 + NT^{1/4}|W|^{21/8} + N^2|W|^{1/2}E(W)^{1/2} \right). \]
Solving this inequality for $E(W)$ gives the result.
\end{proof}

Combining Proposition \ref{S3} and Proposition \ref{EWprop} then gives the following.

\begin{prop}\label{s3final} If $W$ is a $T^\varepsilon$-separated set contained in an interval of length $T$, we have
    \[ S_3 \lessapprox T^2 |W|^{3/2} + B^2 T |W| N^{3 - 2 \sigma} + BT |W|^2 N^{3/2 - \sigma} + BT^{9/8} |W|^{29/16} N^{3/2 - \sigma}. \]
\end{prop}

\begin{proof}[Proof of Proposition \ref{momentGM-local}]

We now assemble the preceding estimates. Cubing Proposition \ref{keyWestimate} and rearranging gives
\begin{equation}\label{final-assembly} |W|^3 B^{3q-6}N^{6\sigma-3} \lessapprox N^3+S_1+S_2+S_3. \end{equation}
By Proposition \ref{s1}, the contribution of $S_1$ is negligible. Applying Proposition \ref{s2} and Proposition \ref{s3final} to bound $S_2$ and $S_3$, respectively, we obtain, for every $k\in\NN$,
\begin{align*} |W|^3 B^{3q-6}N^{6\sigma-3} \lessapprox_{k} {}& N^3 + N^2|W|^2 + TN|W|^{2-1/k} + N^2T^{1/(2k)}|W|^{2-3/(4k)} \\
&+ T^2|W|^{3/2} + B^2T|W|N^{3-2\sigma} + BT|W|^2N^{3/2-\sigma} \\
&+ BT^{9/8}|W|^{29/16}N^{3/2-\sigma}.
\end{align*}
Rearranging term by term, we find
\begin{align}\label{final-W-bound}
|W|\lessapprox_{k} {}& B^{2-q}N^{2-2\sigma} + B^{6-3q}N^{5-6\sigma} + B^{\frac{3k(2-q)}{k+1}} T^{\frac{k}{k+1}} N^{\frac{k(4-6\sigma)}{k+1}} \\
&+ B^{\frac{12k(2-q)}{4k+3}} T^{\frac{2}{4k+3}} N^{\frac{4k(5-6\sigma)}{4k+3}} \nonumber + B^{4-2q}T^{4/3}N^{2-4\sigma} + B^{4-3q/2}T^{1/2}N^{3-4\sigma} \nonumber \\ 
&+ B^{7-3q}TN^{9/2-7\sigma} + B^{\frac{16(7-3q)}{19}} T^{18/19}N^{72/19-112\sigma/19}. \nonumber
\end{align}
By construction, we have $B\geq 1$. Since $q\geq8/3$, every exponent of $B$ appearing on the right-hand side of \eqref{final-W-bound} is nonpositive. Indeed, the strongest restriction is
\[ 4-\frac{3q}{2}\leq0, \]
which is equivalent to $q\geq8/3$. We may therefore discard the powers of $B$ in \eqref{final-W-bound}. Thus,
\begin{align}\label{GM-final-shape} |W|\lessapprox_k & N^{2-2\sigma} + N^{5-6\sigma} + T^{\frac{k}{k+1}} N^{\frac{k(4-6\sigma)}{k+1}} \\
&+ T^{\frac{2}{4k+3}} N^{\frac{4k(5-6\sigma)}{4k+3}} + T^{4/3}N^{2-4\sigma} + T^{1/2}N^{3-4\sigma} \nonumber\\
&+ TN^{9/2-7\sigma} + T^{18/19}N^{72/19-112\sigma/19}. \nonumber
\end{align}
We now take $k=4$ and recall that $T=N^{6/5}$. As in \cite[Section~12]{GM}, \eqref{GM-final-shape} becomes
\begin{align*} |W| \lessapprox T\bigg( & N^{(4-10\sigma)/5} + N^{(19-30\sigma)/5} + N^{(74-120\sigma)/25} + N^{(298-480\sigma)/95} \\
&+ N^{(12-20\sigma)/5} + N^{(9-14\sigma)/2} + N^{(354-560\sigma)/95} \bigg). \end{align*}
The terms on the right are dominated by
\[ T\left( N^{(4-10\sigma)/5} + N^{(12-20\sigma)/5} + N^{(9-14\sigma)/2} \right).\]
For $7/10\leq\sigma\leq4/5$, the first and third terms are dominated by the second. This proves Proposition \ref{momentGM-local} and, in turn, Theorem \ref{momentGM}.
\end{proof}

\begin{rmk}\label{longrmk}
Note that the exponent remains $5/2$ in Theorem \ref{mainautomorphic} even under the full Ramanujan conjecture. Thus, for the purposes of improving Theorem \ref{mainautomorphic}, it would not be necessary or helpful to push these techniques to their limit. It could, however, reduce the required moment information in Theorem \ref{momentGM}. We therefore examine how much we can gain from the sparsity of the coefficients in the dyadic decomposition. There are two places where the specific coefficients $c_n$ are relevant in Guth and Maynard's argument: the initial linear-algebra set-up and the additive energy (see \cite[Lemma 11.4]{GM}). We already took full advantage of the former.

In order to benefit from the latter, we would need for the sparsity of the coefficients to produce a meaningful refinement in the upper bound
\[ E(W_B)  \lessapprox N^{-2 \sigma} \sum_{n_1, n_2 \sim N} c_{n_1} \overline{c_{n_2}} \int_{|s| \lessapprox 1} \left( \frac{n_2}{n_1} \right)^{is} R \left( \frac{n_1}{n_2} \right)^2 R \left( \frac{n_2}{n_1} \right) ds. \]
Thus, a natural place to look is the cubic moment estimate for $R$ in Lemma 11.4 of \cite{GM}. The usual interpolation between the second- and fourth-moment estimates for $R$ mixes two different arithmetic regimes: the second moment primarily sees the spacing of many reduced ratios, while the fourth moment is sensitive to the multiplicity with which a reduced ratio is represented. Unfortunately, we do not know much about the support of the coefficients $c_n$ with $|c_n| \sim B$ other than its size. 

Now work on a subinterval of $[0,T]$ at the local scale $T_{\mathrm{loc}}$. Put $\Delta = N^2/T_{\mathrm{loc}}$ and consider
\[ E(W_B) \leq B^2N^{-2 \sigma} \sum_{d \leq \Delta} \sum_{\substack{n_1', n_2' \sim N/d \\ \gcd(n_1', n_2') = 1 \\ dn_1', dn_2' \in \mathcal{I}_B}} \left\vert R \left( \frac{n_1'}{n_2'} \right) \right\vert^3 + 
B^2N^{-2 \sigma} \sum_{d \geq \Delta} \sum_{\substack{n_1', n_2' \sim N/d \\ dn_1', dn_2' \in \mathcal{I}_B}} \left\vert R \left( \frac{n_1'}{n_2'} \right) \right\vert^3. \]
In the small-gcd range, the ratios $n_1'/n_2'$ are well separated, but there could be many short ratio intervals for which $R(n_1'/n_2')$ is large. In the large-gcd range, the contribution of a fixed reduced ratio $a/b$ is controlled by how often both $da$ and $db$ lie in $\mathcal I_B$ across $d > \Delta$, and there could be many instances of the same fixed reduced ratio.

A counting heuristic indicates that we would need $|\mathcal{I}_B| \ll N^2/T_{\mathrm{loc}}$ to rule out either conspiracy. Given \eqref{symsquare}, this would require $B \gtrapprox (T_{\mathrm{loc}}/N)^{1/4}$. On the other hand, one can show that we would need to take $B \lessapprox N^{1/40}$ in order to improve on Ingham--Huxley in the critical case. Taking the Guth--Maynard local scale $T_{\mathrm{loc}} = N^{6/5}$, these inequalities cannot hold simultaneously. Thus, a cardinality bound for $|\mathcal I_B|$ alone is not able to prevent concentration either in short ratio intervals or along common dilates. This level of distributional control within $\mathcal{I}_B$ appears to be out of reach.
\end{rmk}

\section{Proof of Theorems \ref{main} and \ref{mainautomorphic}}\label{mainproof}
In the zero detection framework of Section \ref{zeroframework}, it remains to optimize $Y$ so the estimates for $N_0 + N_1 + N_2$ are sharp. By Lemma \ref{n0} and Proposition \ref{n2}, for fixed $0 < \eta <1/2$, we have
\begin{equation}\label{n0n2} N_0 + N_2 \ll_\eta (\log T)^2 + T^{1 + 3 \eta} Y^{1 - 2 \sigma}.\end{equation}
Proposition \ref{n1} shows that it suffices to bound the cardinality of a $1$-separated set $W\subset[0,T]$  such that there exist an integer $N$ with
\[ \lfloor Y^{1/2} \rfloor \le N\le CY\log T, \]
and coefficients $c_n$ satisfying 
\[ \sum_{n \sim N} |c_n|^q \ll_\eta N^{1 + o_\eta(1)} \] for which $D(t):=\sum_{n \sim N} c_n n^{it}$ satisfies $ |D(t)|\ge N^{\sigma}$ for all $t \in W$. In particular, we have $N_1\lessapprox_\eta |W|$. By H\"older's inequality,
\[ \sum_{n\sim N}|c_n|^2 \ll_\eta N^{1+o_\eta(1)}.\]
Thus the classical estimates of Proposition \ref{original} apply after a harmless $N^{o_\eta(1)}$ rescaling of the coefficients. Since $N \leq CY \log T \lessapprox T^6$, the resulting loss is absorbed into $\lessapprox_\eta$.

The case $\sigma=\frac{1}{2}$ follows immediately from Lemma \ref{standardzerocount}, so we assume that $\sigma>\frac{1}{2}$. We first treat the range $\frac{1}{2} < \sigma \leq \frac{7}{10}$. Take $Y=T^{\frac{2}{3-2\sigma}}.$
By \eqref{n0n2}, we have
\[ N_0 + N_2 \ll_\eta (\log T)^2 + T^{1 + 3 \eta} Y^{1 - 2 \sigma} \ll_\eta T^{1 + 3\eta + \frac{2(1-2\sigma)}{3-2\sigma}} = T^{\frac{5-6\sigma}{3-2\sigma} + 3\eta}, \]
and we obtain
\[ N_0+N_2 \ll_\delta T^{\frac{4(1-\sigma)}{3-2\sigma}+\delta} \]
given any $\delta > 0$ after taking $\eta < \delta/6$. Thus, applying the mean value theorem \eqref{ingham} to $D(t)$ and using $Y^{1/2} \ll \lfloor Y^{1/2} \rfloor \leq N \leq CY \log T$, we obtain
\[ |W| \lessapprox_\eta N^{2 - 2 \sigma} + TN^{1 - 2 \sigma}
\lessapprox_\eta Y^{2 - 2 \sigma} + TY^{1/2 - \sigma}. \]
Now we have
\[ Y^{2 - 2 \sigma} = T^{\frac{4(1-\sigma)}{3 -2\sigma}}, \qquad
TY^{1/2 - \sigma} = T^{1+\frac{(1-2\sigma)}{3-2\sigma}} = T^{\frac{4(1-\sigma)}{3-2\sigma}}.\]
Combining this with $N_1 \lessapprox_\eta |W|$, we obtain  \[ N_1 \lessapprox_\eta T^{\frac{4(1-\sigma)}{3 - 2 \sigma}}.\]
Together with the estimate for $N_0+N_2$ and with the fixed choice of $\eta < \delta/6$, this gives
\[ N_L(\sigma,T) \ll T^{\frac{4(1-\sigma)}{3-2\sigma} + o(1)}. \]
Thus, we may take
\[ \alpha(\sigma) = \frac{4}{3 - 2 \sigma} \quad \text{ given } \quad \frac{1}{2} \leq \sigma \leq \frac{7}{10}.\]

We now turn to the middle range. We take $Y = T$. By \eqref{n0n2}, we have that 
\begin{equation}\label{bound} N_0 + N_2 \ll_\eta (\log T)^2 + T^{2 - 2 \sigma + 3 \eta} \ll_\eta T^{2 - 2 \sigma + 3 \eta} \ll_\delta T^{2 - 2 \sigma + \delta} \end{equation}
given any $\delta > 0$ after taking $\eta < \delta/6$. If $T^{5/8} \le N \le T^{\frac{25(1-\sigma)}{4(9 - 10 \sigma)}}$, we apply Theorem \ref{momentGM} to the Dirichlet polynomial $D(t)$. (Here, the factor $(NT)^{o(1)}$ in Theorem \ref{momentGM} is absorbed into $\lessapprox_\eta$.) We note that $\sigma \in [7/10, 4/5]$ implies that $N^{2 - 2\sigma}$ is dominated by $N^{18/5 - 4 \sigma}$ so that 
\begin{equation}\label{ineq1} |W| \lessapprox_\eta (T^{\frac{25(1-\sigma)}{4(9 - 10 \sigma)}})^{18/5 - 4 \sigma} + T(T^{5/8})^{12/5 - 4 \sigma} \lessapprox_\eta T^{5/2(1 - \sigma)}. \end{equation}

We check that this estimate holds in the remaining ranges using the classical estimates of Proposition \ref{original}. If $\lfloor T^{1/2} \rfloor \le N \le T^{5/8}$, we can apply the mean value theorem as in \eqref{ingham} to the Dirichlet polynomial $D(t)^2$ so that 
\begin{equation}\label{ineq2} |W| \lessapprox_\eta N^{4 - 4 \sigma} \lessapprox_\eta T^{5/2(1 - \sigma)}.\end{equation}
Further, if $T^{\frac{25(1-\sigma)}{4(9 - 10 \sigma)}} \le N \ll Y \log T = T \log T$, we can apply the mean value estimate as in \eqref{ingham} to $D(t)$ so that 
\begin{equation}\label{ineq3} |W| \lessapprox_\eta T (T^{\frac{25(1-\sigma)}{4(9 - 10 \sigma)}})^{1-2 \sigma} = T^{(61 - 115 \sigma + 50 \sigma^2)/(36-40 \sigma)}. \end{equation}
We can quickly verify that this is less than $T^{5/2(1 - \sigma)}$. In \cite{GM}, it is shown that if $N \lessapprox T^{1/2}$, we have $|W| \lessapprox_\eta T^{15(1-\sigma)/(3+5 \sigma)}$, which is also less than $T^{5/2(1 - \sigma)}$ for $\sigma \in [7/10, 4/5]$.
Combining these estimates with $N_1\lessapprox_\eta |W|$, we obtain
\[ N_1 \lessapprox_\eta T^{\frac{5}{2}(1-\sigma)}. \]
Together with \eqref{bound}, and with the fixed choice $\eta<\delta/6$, this gives
\[ N_L(\sigma,T) \ll T^{\frac{5}{2}(1-\sigma)+o(1)}, \] so we may take $\alpha(\sigma)=\frac{5}{2}$ for all $\frac{7}{10}\leq\sigma\leq\frac{4}{5}$.

It remains to treat the range $\frac{4}{5} \leq \sigma \leq 1$. We now take $Y=T^{1/\sigma}$. By \eqref{n0n2}, we have
\[ N_0 + N_2 \ll_\eta (\log T)^2 + T^{1 + 3\eta}Y^{1-2\sigma} \ll_\eta T^{1 + 3\eta + \frac{1-2\sigma}{\sigma}} \ll_\delta  T^{\frac{2(1-\sigma)}{\sigma}+\delta}\]
for all $\delta > 0$ after choosing $\eta < \delta/6$. For $N_1$, we apply Huxley's estimate \eqref{huxley} to $D(t)$:
\[ |W| \lessapprox_\eta  N^{2-2\sigma} + TN^{4-6\sigma} \lessapprox_\eta Y^{2- 2 \sigma} + T Y^{2 - 3\sigma}. \]
Now, we have 
\[ Y^{2 - 2 \sigma} = T^{\frac{2(1-\sigma)}{\sigma}}, \qquad TY^{2 - 3 \sigma} = T^{1 + \frac{2 - 3 \sigma}{\sigma}} = T^{\frac{2(1-\sigma)}{\sigma}}. \]
Combining this with $N_1 \lessapprox_\eta |W|$ and with the fixed choice $\eta < \delta/6$, this gives 
\[ N_L(\sigma,T) \ll T^{\frac{2}{\sigma} (1-\sigma)+o(1)},\]
so we may take
\[ \alpha(\sigma) = \frac{2}{\sigma} \quad \text{ given } \quad \frac{4}{5} \leq \sigma \leq 1. \]
This proves Theorem \ref{main}. Theorem \ref{mainautomorphic} follows from Theorem \ref{main} and the discussion in Sections \ref{automorphicprop} and \ref{properties}. \qed

\begin{rmk} The exponent $5/2$ is determined by the critical length $N \asymp T^{5/8}$, at which the mean-value bound applied to $D(t)^2$ and the Guth--Maynard bound balance. 
\end{rmk}

\begin{rmk}\label{qinterpolation}
The preceding proof also gives intermediate zero density estimates when $\frac{5}{2}<q<\frac{8}{3}$. Indeed, the bound in \eqref{final-W-bound} contains the factor $B^{4-\frac{3q}{2}}$. For $q \geq 8/3$, this factor is largest when $B \asymp 1$, while for $q<8/3$, it is largest at the maximal admissible scale $B \lessapprox N^{1/q}$. Carrying this loss through the preceding optimization gives
\[ N_L(\sigma,T) \lessapprox T^{\frac{40q}{31q-40}(1-\sigma) + o(1)} \]
for $\frac{7}{10} < \sigma < \frac{4}{5}$, which improves the classical estimates for $\frac{2}{q}-\frac{1}{20} < \sigma < \frac{31}{20}-\frac{2}{q}$. As $q$ increases from $5/2$ to $8/3$, the uniform exponent decreases from $8/3$ to $5/2$, and the interval of improvement expands from the point $\sigma=3/4$ to $\sigma \in [7/10,4/5]$. This identifies $q>5/2$ as the threshold for improving the classical zero density estimates and $q\geq8/3$ as the threshold for recovering the exponent $5/2$ of Theorem \ref{main}.
\end{rmk}

\section{Short-interval consequences}\label{applications}

For a cuspidal automorphic representation $\pi$ of $\GL_2(\AAA_{\QQ})$ with unitary central character, define the generalized von Mangoldt coefficients $\Lambda_\pi(n)$ by 
\[ -\frac{L'}{L}(s, \pi) = \sum_{n=1}^\infty \frac{\Lambda_\pi(n)}{n^s}, \qquad \Real(s) > 1.\]
Broadly speaking, a zero density estimate with exponent $c$ gives a short-interval estimate in the range $h\geq x^{1-1/c+\varepsilon}$ for any $\varepsilon > 0$, provided that it is accompanied by suitable quantitative control near $\sigma=1$ and a sufficiently wide zero-free region, such as the Vinogradov--Korobov region \eqref{VK-hypothesis}.

The subpolynomial loss in Theorem \ref{mainautomorphic} is harmless when $\sigma$ is bounded away from $1$, but the standard Hoheisel argument requires quantitative control of this loss when $\sigma$ is close to $1$. For this range, we use a result of Heath-Brown \cite{HB}. Although his result is stated for Dedekind zeta functions of quadratic fields, one can verify that his argument applies more generally to admissible degree $2$ $L$-functions with $q$th-moment input for any $q \geq 2$; the details are omitted.

\begin{lem}\label{nearone} Let $L(s)$ be an admissible degree $2$ $L$-function with $q$th-moment input for some $q \geq 2$. For every fixed $\delta > 0$, there exists $A = A(L, \delta) > 0$ such that 
\begin{equation}
    N_L(\sigma, T) \ll_{L, \delta} T^{2(1 + \delta)(1-\sigma)} (\log T)^A
\end{equation}
uniformly for $9/10 \leq \sigma \leq 1$.
\end{lem}

We combine Lemma \ref{nearone} with Theorem \ref{mainautomorphic} to obtain the following short-interval estimate.

\begin{prop}\label{shortintervalLambda}
Let $\pi$ be a cuspidal automorphic representation of $\GL_2(\AAA_{\QQ})$ with unitary central character, and suppose that $L(s,\pi)$ satisfies \eqref{VK-hypothesis}. Then, for every fixed $\varepsilon>0$,
\[ \sum_{x<n\leq x+h}\Lambda_\pi(n) \ll_{\pi, \varepsilon} h \exp \bigl(-(\log x)^{1/4} \bigr) \]
uniformly for $x^{3/5+\varepsilon}\leq h\leq x$ as $x\to\infty$.
\end{prop}

\begin{proof}
The explicit formula (see \cite[Chapter~10]{IK}) gives
\[ \sum_{x<n\leq x+h}\Lambda_\pi(n) = -\sum_{\substack{\rho\\|\gamma|\leq T}} \frac{(x+h)^\rho-x^\rho}{\rho} + O_\pi\left(\frac{x(\log x)^3}{T}\right). \]
Following \cite[Section 13.2]{GM}, we take $T=\frac{x}{h}\exp\bigl(2(\log x)^{1/4}\bigr)$. Separating the real parts of the zeros into intervals of length $1/\log x$ gives
\begin{equation}\label{explicit-formula-short} \left|\sum_{x<n\leq x+h}\Lambda_\pi(n)\right| \ll_\pi h(\log x)\sup_{\frac{1}{2} \leq \sigma \leq 1} x^{\sigma-1}N_\pi(\sigma,T) + h\exp\bigl(-(\log x)^{1/4}\bigr).
\end{equation}
Since $h\geq x^{3/5+\varepsilon}$, we have $T\leq x^{2/5-\varepsilon/2}$ for sufficiently large $x$. For $\sigma\leq9/10$, Theorem~\ref{mainautomorphic} gives
\[ \sup_{1/2 \leq \sigma \leq 9/10} x^{\sigma-1}N_\pi(\sigma,T) \ll_{\pi} \sup_{1/2 \leq \sigma \leq 9/10} \left(\frac{T^{5/2+o(1)}}{x}\right)^{1-\sigma} \ll_{\pi, \varepsilon} x^{-c_\varepsilon}  \]
for some $c_\varepsilon > 0$. For $\sigma\geq9/10$, choose $\delta>0$ sufficiently small in terms of $\varepsilon$. By Lemma \ref{nearone},
\[ x^{\sigma-1}N_\pi(\sigma,T) \ll_{\pi,\delta} (\log T)^A \left(\frac{T^{2(1+\delta)}}{x}\right)^{1-\sigma}. \]
The zero-free region \eqref{VK-hypothesis} implies that $N_\pi(\sigma, T)$ vanishes unless
\[ 1-\sigma \geq \frac{c}{(\log T)^{2/3}(\log\log T)^{1/3}} \]
for some $c > 0$ as in \eqref{VK-hypothesis}. Since $T^{2(1+\delta)}/x\leq x^{-c_\varepsilon}$, it follows that
\[ \sup_{9/10 \leq \sigma \leq 1} x^{\sigma-1}N_\pi(\sigma,T) \ll_{\pi,\varepsilon} (\log T)^A \exp\left( -c_\varepsilon c \frac{\log x} {(\log T)^{2/3}(\log\log T)^{1/3}} \right) \ll_{\pi,\varepsilon} \exp\bigl(-2(\log x)^{1/4}\bigr). \]
Substituting these estimates into \eqref{explicit-formula-short} completes the proof.
\end{proof}

At every unramified prime $p$, we have $\Lambda_\pi(p)=\lambda_\pi(p)\log p$. Moreover, using the Kim--Sarnak bound \eqref{satakebound}, we can show that the contribution of higher prime powers and ramified primes is absorbed into the error term uniformly in the range of Proposition \ref{shortintervalLambda}, and Corollary \ref{shortintervalcoefficients} follows.

For dihedral automorphic $L$-functions, a zero-free region of the form \eqref{VK-hypothesis} is known due to \cite{Coleman}. By the modularity theorem, every CM elliptic curve $E/\QQ$ is associated to a cuspidal automorphic representation $\pi_E$ of $\GL_2(\AAA_{\QQ})$, and the theory of complex multiplication shows that $\pi_E$ is dihedral. Moreover, we have
\[ \lambda_{\pi_E}(p)=\frac{a_E(p)}{\sqrt p} \]
for every $p\nmid N_E$, and Corollary \ref{frobeniustraces} follows.

\section{Dedekind zeta functions and Chebotarev applications}\label{otherdegree2}

Let $K/\QQ$ be a quadratic field of \mbox{discriminant $D$}. Let $\zeta_K(s)$ be the associated Dedekind zeta function; note that $\zeta_K$ satisfies the Ramanujan \mbox{conjecture}. The familiar factorization $\zeta_K(s)=\zeta(s)L(s,\chi_D)$, where $\chi_D$ is the primitive quadratic Dirichlet character associated to $K$, gives
\[ N_{\zeta_K}(\sigma,T) = N_{\zeta}(\sigma,T) + N_{L(s,\chi_D)}(\sigma,T).\]
It follows from \eqref{degoneZDE} that \begin{equation}\label{quaddedekind} N_{\zeta_K}(\sigma,T) \ll T^{\frac{30}{13}(1-\sigma)+o(1)}. \end{equation}

It is useful to distinguish this result from a direct treatment of $\zeta_K$ as a degree $2$ $L$-function. Heath-Brown \cite[Theorem~2]{HB} proved the following zero density estimate for Dedekind zeta functions of quadratic number fields:
\[ N_{\zeta_K}(\sigma,T) \ll T^{\frac{8}{3}(1-\sigma)+o(1)}. \]
Applying our methods directly to $\zeta_K$ would replace $8/3$ by $5/2$, improving this result. However, the problem reduces to the degree $1$ case, and the stronger exponent $30/13$ follows.\footnote{Heath-Brown observed in the appendix of \cite{HB} that the factorization of $\zeta_K$ permits degree $1$ estimates to be transferred to quadratic Dedekind zeta functions. Nonetheless, the exponent $8/3$ has remained in subsequent applications involving zero density estimates. This is the input used, for instance, by Balog and Ono \cite{BO}.}

Cubic fields provide the first genuinely degree $2$ case: if $K/\QQ$ is a non-Galois cubic field, then
\[ \zeta_K(s)=\zeta(s)L(s,\varrho_K), \]
where $\varrho_K$ is an irreducible two-dimensional Artin representation. More generally, suppose that a Dedekind zeta function factors into Artin $L$-functions of degree at most $2$ and that each degree $2$ factor is holomorphic. Since an Artin representation has finite image, its unramified Frobenius eigenvalues are roots of unity; its coefficients therefore satisfy the Ramanujan bound. Thus, \mbox{Theorem}~\ref{main} applies without requiring the degree $2$ factors to be automorphic. This gives an improved short-interval Chebotarev range, sharpening those obtained from the short-interval Chebotarev theorems of Balog and Ono \cite{BO} and Thorner \cite[Corollary~1.2]{Thorner} for such extensions.

We define the counting function $\pi_C(X; K/K')$ for a number field $K'$, a normal extension $K/K'$, and a conjugacy class $C \subset \Gal(K/K')$ by
\small
\vspace{-0.5em}
\begin{equation} \pi_C(X; K/K')  := \# \{ \mathfrak{P} \subseteq \mathcal{O}_{K'}: \mathfrak{P} \text{ prime and unramified in } \mathcal{O}_K, \Big[ \frac{K/K'}{\mathfrak{P}} \Big]  = C, N_{K'/\QQ}(\mathfrak{P}) \leq X \} \end{equation}
\normalsize
{\spaceskip=0.7\fontdimen2\font where $\Big[ \frac{K/K'}{\mathfrak{P}} \Big]$ denotes the Artin symbol for the conjugacy class of the Frobenius above $\mathfrak{P}$ in $\Gal(K/K')$.}

As in the proof of Proposition \ref{shortintervalLambda}, the subpolynomial losses in the zero density estimates are harmless away from $\sigma = 1$. Near $\sigma = 1$, we use Lemma \ref{nearone} for the degree $2$ factors and Huxley's classical estimate for the degree $1$ factors. Additionally, $\zeta_K$ satisfies a Vinogradov--Korobov zero-free region \cite[Lemma 3.3]{BO}. Combining these inputs with the explicit formula underlying the effective Chebotarev density theorem \cite{LO} and Theorem \ref{main} gives the following.

\begin{cor}\label{dedekind} Let $K$ be a number field, $K/K'$ be a finite Galois extension, and $C \subseteq \Gal(K/K')$ be a conjugacy class. Suppose that, for $e_i, f_j \in \NN$,
\begin{equation}\label{lowdegreefactorization}
\zeta_K(s) = \prod_{i=1}^{r_1}L(s,\chi_i)^{e_i} \prod_{j=1}^{r_2}L(s,\varrho_j)^{f_j},
\end{equation} where the $\chi_i$ are Dirichlet characters and the $\varrho_j$ are irreducible two-dimensional Artin representations such that $L(s, \varrho_j)$ is holomorphic.
For every $\varepsilon > 0$,
\[ \pi_C(X+Y;K/K')-\pi_C(X;K/K') = \left( \frac{|C|}{|\Gal(K/K')|} + o(1) \right) \frac{Y}{\log X} \]
holds, as $X \rightarrow \infty$, uniformly for
\[ X^{3/5+\varepsilon}\leq Y\leq X. \]
If $r_2=0$, then the same asymptotic holds uniformly for $X^{17/30+\varepsilon}\leq Y\leq X$.
\end{cor}

Corollary~\ref{dedekind} applies unconditionally to Galois extensions $K/\QQ$ whose Galois groups are \mbox{generalized} dihedral groups; that is,
\[ \Gal(K/\QQ)\cong A\rtimes (\ZZ/2\ZZ) \]
where $A$ is abelian and $1 \in \ZZ/2\ZZ$ acts on $A$ by inversion. For such $K$, every irreducible representation has degree at most $2$, and the corresponding degree $2$ Artin $L$-functions are entire Hecke $L$-functions over the quadratic subfield $K^A$. 

As a concrete example, fix $n \in \NN$, and let $H_{-4n}$ be the ring class field of the order $\ZZ[\sqrt{-n}]$ in $\QQ(\sqrt{-n})$. Then $H_{-4n}/\QQ$ is generalized dihedral \cite[Lemma 9.3]{Cox}, with
\[ \Gal(H_{-4n}/\QQ) \cong \Gal(H_{-4n}/\QQ(\sqrt{-n})) \rtimes \ZZ/2\ZZ \cong \Cl(-4n)\rtimes \ZZ/2 \ZZ, \]
where $\ZZ/2\ZZ$ acts on the ring class group by inversion. Moreover, for every prime $p\nmid 2n$, $p$ is of the form $x^2+ny^2$ for $x,y \in \NN$ if and only if $p$ splits completely in $H_{-4n}$; see \cite[Theorem 9.4]{Cox}. Applying Corollary~\ref{dedekind} with the identity conjugacy class therefore gives
\[ \#\{X<p\leq X+Y:p \text{ is of the form } x^2+ny^2\} = \left(\frac{1}{2h(-4n)}+o(1)\right)\frac{Y}{\log X}\]
uniformly for $X^{3/5+\varepsilon}\leq Y\leq X$ for all $\varepsilon > 0$, where $h(-4n)=|\Cl(-4n)|$.

Finally, we record when the holomorphy hypothesis in Corollary~\ref{dedekind} is known. If
\[ \varrho:\Gal(\overline{\QQ}/\QQ) \longrightarrow \GL_2(\CC) \]
is an irreducible two-dimensional Artin representation, then the projective image of $\varrho$ is dihedral, tetrahedral, octahedral, or icosahedral; equivalently, it is isomorphic to a dihedral group, $A_4$, $S_4$, or $A_5$ \cite[Introduction]{BDST}. In the dihedral case, the corresponding Artin $L$-function is a Hecke $L$-function over a quadratic field, so both its holomorphy and automorphy are classical \cite{Artin2, Gelbart}. Further, the tetrahedral and octahedral cases are automorphic by the work of Langlands and Tunnell \cite{Langlands, Tunnell}. Moreover, every odd two-dimensional Artin representation over $\QQ$, including those of icosahedral type, is automorphic as a consequence of the proof of Serre's modularity conjecture \cite{KW1, KW2, Kisin}. Thus, the only unresolved case is that of an even icosahedral factor, and Corollary \ref{dedekind} is unconditional unless such a factor occurs. In this remaining case, its conclusion follows conditionally on Artin's holomorphy conjecture \cite{Artin} for this factor.

\section{Remarks on Higher Degree $L$-Functions}\label{highdeg}

It is natural to ask whether the work of Guth and Maynard can be used to improve zero density estimates for degree $d$ $L$-functions and, in particular, for $\GL_d$ automorphic $L$-functions, when $d > 2$.

From the reflection principle for Dirichlet polynomials, recall that it suffices to consider polynomials of length at most the square root of the analytic conductor. For a fixed degree $d$ $L$-function at height $|t| \asymp T$, this length is $\lessapprox \mathcal{C}_L(t)^{1/2} \asymp T^{d/2}$. For $d=2$, this gives $N \lessapprox T$, which is precisely the range where Guth and Maynard improve on classical large-values bounds. For $d>2$, the lengths of Dirichlet polynomials extend beyond this range, so their result is no longer useful.

Indeed, even if Montgomery's conjecture \eqref{montgomery-conjectural} were available for $N=T^{d/2+o(1)}$, it would only give
\[ |W| \lessapprox N^{2 - 2 \sigma} = T^{(d+o(1))(1-\sigma)} \ll_\varepsilon T^{(d+\varepsilon)(1-\sigma)} \]
for all $\varepsilon > 0$. This is the same exponent $(d+\varepsilon)(1-\sigma)$ appearing in Heath-Brown's work \cite{HB}.\footnote{Although Heath-Brown states his theorem only for Dedekind zeta functions, the argument in fact extends much further: it applies to a broad class of $L$-functions and requires only $\ell^2$ control of the coefficients.} These large-values estimates use no information about the coefficients of the Dirichlet polynomial beyond control of their size and can thus not produce an improved zero density estimate on their own. One would have to instead exploit additional structure specific to the underlying $L$-function. For example, recent improvements for $\GL_3$ automorphic $L$-functions make use of the $\GL_3$ Voronoi summation formula to obtain nontrivial critical-line second-moment estimates; see \cite{DLY, Pal}. This reflects the same barrier to establishing a range of $\sigma$ for which the density hypothesis holds for a general $L$-function of degree greater than $2$.

Note also that the same obstruction arises for $\GL_2$ automorphic $L$-functions over a number field $F \neq \QQ$. For a fixed automorphic representation $\pi$ of $\GL_2(\AAA_F)$, the analytic conductor of $L(s,\pi)$ at height $T$ is of size $T^{2[F:\QQ]}$, so one must  handle Dirichlet polynomials of length up to $T^{[F:\QQ]}$, which exceeds $T$.
\vspace{-1em}

\bibliographystyle{plain}
\bibliography{bib.bib}

@article{Artin,
  author  = {Artin, E.},
  title   = {{\"U}ber eine neue {A}rt von {$L$}-Reihen},
  journal = {Abh. Math. Seminar. Univ. Hambg.},
  volume  = {3},
  year    = {1924},
  pages   = {89--108}
}

@article{BO,
  author  = {Balog, A. and Ono, K.},
  title   = {The {Chebotarev} density theorem in short intervals and some questions of {Serre}},
  journal = {J. Number Theory},
  volume  = {91},
  year    = {2001},
  number  = {2},
  pages   = {356--371}
}

@article{GM,
  author  = {Guth, L. and Maynard, J.},
  title   = {New large value estimates for {Dirichlet} polynomials},
  journal = {Ann. of Math. (2)},
  volume  = {203},
  number  = {2},
  year    = {2026},
  pages   = {623--675}
}

@article{HB,
  author  = {Heath-Brown, D. R.},
  title   = {On the density of the zeros of the {Dedekind} zeta-function},
  journal = {Acta Arith.},
  volume  = {33},
  year    = {1977},
  number  = {2},
  pages   = {169--181}
}

@article{Huxley,
  author  = {Huxley, M. N.},
  title   = {On the difference between consecutive primes},
  journal = {Invent. Math.},
  volume  = {15},
  year    = {1972},
  pages   = {164--170}
}

@article{MontgomeryLargeValues,
  author  = {Montgomery, H. L.},
  title   = {Mean and large values of {Dirichlet} polynomials},
  journal = {Invent. Math.},
  volume  = {8},
  year    = {1969},
  pages   = {334--345}
}

@article{HuxleyLargeValues,
  author  = {Huxley, M. N.},
  title   = {Large values of {Dirichlet} polynomials},
  journal = {Acta Arith.},
  volume  = {24},
  year    = {1973},
  pages   = {329--346}
}

@article{Ingham,
  author  = {Ingham, A. E.},
  title   = {On the estimation of {$N(\sigma,T)$}},
  journal = {Q. J. Math.},
  volume  = {11},
  year    = {1940},
  pages   = {291--292}
}

@incollection{Ivic,
  author    = {Ivi{\'c}, A.},
  title     = {On zeta-functions associated with {Fourier} coefficients of cusp forms},
  booktitle = {Proceedings of the Amalfi Conference on Analytic Number Theory},
  editor    = {Bombieri, E. and others},
  publisher = {Universit{\`a} di Salerno},
  year      = {1992},
  pages     = {231--246}
}

@book{IK,
  author    = {Iwaniec, H. and Kowalski, E.},
  title     = {Analytic Number Theory},
  series    = {Amer. Math. Soc. Colloq. Publ.},
  volume    = {53},
  publisher = {American Mathematical Society},
  address   = {Providence, RI},
  year      = {2004}
}

@article{Jutila,
  author  = {Jutila, M.},
  title   = {On a density theorem of {H. L. Montgomery} for {$L$}-functions},
  journal = {Ann. Acad. Sci. Fenn. Ser. A I},
  volume  = {520},
  year    = {1972}
}

@book{Mont,
  author    = {Montgomery, H. L.},
  title     = {Topics in Multiplicative Number Theory}, series = {Lecture Notes in Mathematics},
  publisher = {Springer-Verlag},
  address   = {Berlin-Heidelberg-New York},
  year      = {1971}, volume = {227}
}

@article{Yashiro,
  author  = {Yashiro, Y.},
  title   = {Zero-density estimates for {$L$}-functions attached to cusp forms},
  journal = {JP J. Algebra Number Theory Appl.},
  volume  = {36},
  year    = {2015},
  pages   = {99--122}
}

@article{Tang,
  author  = {Tang, H.},
  title   = {Zero density of {$L$}-functions related to {Maass} forms},
  journal = {Front. Math. China},
  volume  = {8},
  year    = {2013},
  pages   = {923--932}
}

@article{HB2, author = {Heath-Brown, D. R.}, title = {A large values estimate for {D}irichlet polynomials}, journal = {J. Lond. Math. Soc. (2)}, volume ={20}, number = {1}, year = {1979}, pages = {8--18}}

@article{JPSS,
  author  = {Jacquet, H. and Piatetski-Shapiro, I. I. and Shalika, J. A.},
  title   = {Rankin--{S}elberg convolutions},
  journal = {Amer. J. Math.},
  volume  = {105},
  year    = {1983},
  number  = {2},
  pages   = {367--464}
}

@article{GJ, author = {Gelbart, S. and Jacquet, H.}, title = {A relation between automorphic representations of $\operatorname{GL}(2)$ and $\operatorname{GL}(3)$}, journal = {Ann. Sci. \'{E}c. Norm. Sup\'er. (4)}, volume = {11}, year = {1978}, pages = {471--542}}

@article{KS, author = {Kim, H. and Shahidi, F.}, title = {Functorial products for $\operatorname{GL}_2 \times \operatorname{GL}_3$ and the symmetric cube for $\operatorname{GL}_2$}, journal = {Ann. of Math. (2)}, volume = {155}, year = {2002}, pages = {837--893}, number = {3}
}

@article{Kim, author = {Kim, H.}, title = {Functoriality for the exterior square of $\operatorname{GL}_4$ and the symmetric fourth of $\operatorname{GL}_2$}, journal = {J. Amer. Math. Soc.}, volume = {16}, number = {1}, pages = {139--183}, year = {2003}, note = {with Appendix 1 by Dinakar Ramakrishnan and Appendix 2 by Henry H. Kim and Peter Sarnak.}}

@article{Thorner, author = {Thorner, J.}, title = {A variant of the {B}ombieri--{V}inogradov theorem in short intervals and some questions of {S}erre}, journal = {Math. Proc. Cambridge Philos. Soc.}, volume = {161}, number = {1}, year = {2016}, pages = {53--63}}

@article{BDST, author = {Buzzard, K. and Dickinson, M. and
             Shepherd-Barron, N. and Taylor, R.}, title = {On icosahedral {A}rtin representations}, journal = {Duke Math. J.}, volume = {109}, number = {2}, year = {2001}, pages = {283--318}}

@book{Langlands,
  author    = {Langlands, R.},
  title     = {Base Change for {$\operatorname{GL}(2)$}},
  series    = {Annals of Mathematics Studies},
  volume    = {96},
  publisher = {Princeton University Press},
  address   = {Princeton, NJ},
  year      = {1980}
}

@book{MontgomeryTenLectures,author    = {Montgomery, H. L.},
  title     = {Ten Lectures on the Interface Between Analytic Number
               Theory and Harmonic Analysis},
  series    = {CBMS Regional Conference Series in Mathematics},
  volume    = {84},
  publisher = {American Mathematical Society},
  address   = {Providence, RI},
  year      = {1994}
}

@article{Coleman,
  author  = {Coleman, M. D.},
  title   = {A zero-free region for the {Hecke} {$L$}-functions},
  journal = {Mathematika},
  volume  = {37},
  year    = {1990},
  number  = {2},
  pages   = {287--304}
}

@article{Deligne,
  author  = {Deligne, P.},
  title   = {La conjecture de {Weil}. {I}},
  journal = {Publ. Math. Inst. Hautes \'{E}tudes Sci.},
  volume  = {43},
  year    = {1974},
  pages   = {273--307}
}

@article{Tunnell, author = {J. Tunnell}, title = {Artin's conjecture for representations of octahedral type}, journal = {Bull. Amer. Math. Soc.}, volume = {5}, number = {2}, year = {1981}}

@article{Artin2, author = {E. Artin}, title = {Zur {T}heorie der {L}-{R}eihen mit allgemeinen {G}ruppencharakteren}, journal = {Abh. Math. Semin. Univ. Hambg.}, volume ={8}, year ={1931}, pages = {292--306}}

@article{KW1, author = {C. Khare and J.P. Wintenberger}, title = {Serre's modularity conjecture~{I}}, journal = {Invent. Math.}, volume = {178}, number ={3}, year = {2009}, pages = {485--504}}

@article{KW2, author = {C. Khare and J.P. Wintenberger}, title = {Serre's modularity conjecture~{II}}, journal = {Invent. Math.}, volume = {178}, number = {3}, year = {2009}, pages = {505--586}}

@article{Kisin, author = {M. Kisin}, title = {Modularity of $2$-adic {B}arsotti--{T}ate representations}, journal = {Invent. Math.}, volume = {178}, number = {3}, pages = {587--634}, year = {2009}}

@misc{HumphriesAutomorphicNotes,
  author       = {P. Humphries},
  title        = {Automorphic Forms and Automorphic Representations},
  howpublished = {Lecture notes for MATH 8510, University of Virginia, available from the author's teaching page},
  year         = {2023}}

@book{Gelbart,
  author    = {Gelbart, S.},
  title     = {Automorphic Forms on Adele Groups},
  series    = {Annals of Mathematics Studies},
  volume    = {83},
  publisher = {Princeton University Press},
  address   = {Princeton, NJ},
  year      = {1975}
}

@book{GodementJacquet,
  author    = {Godement, R. and Jacquet, H.},
  title     = {Zeta Functions of Simple Algebras},
  series    = {Lecture Notes in Mathematics},
  volume    = {260},
  publisher = {Springer-Verlag},
  address   = {Berlin},
  year      = {1972}
}

@inbook{LO,
author = {J. Lagarias and A. Odlyzko},
title = {Effective versions of the {C}hebotarev density theorem},
series = {Algebraic {n}umber {f}ields: $L$-functions and {G}alois {p}roperties},
year = {1977},
publisher = {Academic Press, London},
pages = {409--464},
}

@misc{DLY,
  author       = {Dasgupta, A. and Leung, W. H. and Young, M. P.},
  title        = {The second moment of the {$\mathrm{GL}_3$} standard
                  {$L$}-function on the critical line},
  howpublished = {arXiv:2407.06962}
}

@article{Pal,
  author  = {Pal, S.},
  title   = {Second moment of degree three {$L$}-functions},
  journal = {Int. Math. Res. Not. IMRN},
  volume  = {2025},
  year    = {2025},
  number  = {15},
  pages   = {1--37}
}

@book{Cox,
  author    = {Cox, D.},
  title     = {Primes of the Form $x^2+ny^2$: Fermat, Class Field
               Theory, and Complex Multiplication},
  edition   = {2},
  publisher = {John Wiley \& Sons},
  address   = {Hoboken, NJ},
  year      = {2013}
}

@article{HBconvexity,
  author  = {Heath-Brown, D. R.},
  title   = {Convexity bounds for {$L$}-functions},
  journal = {Acta Arith.},
  volume  = {136},
  number  = {4},
  year    = {2009},
  pages   = {391--395}
}

@article{MichelVenkatesh,
  author  = {Michel, P. and Venkatesh, A.},
  title   = {The subconvexity problem for {$\GL_2$}},
  journal = {Publ. Math. Inst. Hautes \'{E}tudes Sci.},
  volume  = {111},
  year    = {2010},
  pages   = {171--271}
}

@article{SankaranarayananSengupta,
  author  = {Sankaranarayanan, A. and Sengupta, J.},
  title   = {Zero-density estimate of {$L$}-functions attached to {Maass} forms},
  journal = {Acta Arith.},
  volume  = {127},
  year    = {2007},
  number  = {3},
  pages   = {273--284}
}

\end{document}